\documentclass[oneside,article]{memoir}

\usepackage{amsmath}         % align-like environments
\usepackage{amsthm,thmtools,thm-restate} % theorem-like environments
\usepackage{enumitem}        % enumerate-like environments

\usepackage{rotating}

\usepackage[final]{listings}
\usepackage{bbm}
\usepackage{xcolor}

\definecolor{codegreen}{rgb}{0,0.6,0}
\definecolor{codegray}{rgb}{0.5,0.5,0.5}
\definecolor{codepurple}{rgb}{0.58,0,0.82}
\definecolor{backcolour}{rgb}{0.95,0.95,0.92}

\lstdefinestyle{mystyle}{
    backgroundcolor=\color{backcolour},   
    commentstyle=\color{codegreen},
    keywordstyle=\color{magenta},
    numberstyle=\tiny\color{codegray},
    stringstyle=\color{codepurple},
    basicstyle=\ttfamily\footnotesize,
    breakatwhitespace=false,         
    breaklines=true,                 
    captionpos=b,                    
    keepspaces=true,                 
    numbers=left,                    
    numbersep=5pt,                  
    showspaces=false,                
    showstringspaces=false,
    showtabs=false,                  
    tabsize=2
}

\usepackage{algorithm}
\usepackage{algpseudocode}
\usepackage{algorithmicx}
\usepackage{threeparttable}
\usepackage{booktabs}
\usepackage{float}

\usepackage{tikz}
\usetikzlibrary{matrix,arrows,positioning}
\usepackage{amsfonts}
\usepackage{parskip}
\usepackage{mathtools}
\usepackage{wrapfig}

\usepackage{tikz-cd}

\usepackage{caption, subcaption} % for subfigures

\usepackage[hidelinks]{hyperref}                   % hyperlinks
\usepackage[capitalize,nosort]{cleveref}           % named references

\usepackage{amssymb}
\usepackage[mathscr]{euscript}
\usepackage{relsize}
\usepackage{stmaryrd}
\usepackage{stackengine}

\usepackage{indentfirst} % indent the first paragraph
\usepackage{multicol}

\usepackage{graphicx}
\usepackage{ragged2e} % for nicer left alignment inside columns
\usepackage[inline,nomargin,draft]{fixme}

\isopage[12]

\setlrmargins{*}{*}{1}
\checkandfixthelayout

\counterwithout{section}{chapter}
\usepackage{appendix}

\makeatletter
\providecommand{\abx@aux@refcontext}[1]{}
\providecommand{\abx@aux@cite}[2]{}
\providecommand{\abx@aux@segm}[3]{}
\providecommand{\abx@aux@page}[2]{}
\providecommand{\abx@aux@fnpage}[2]{}
\providecommand{\abx@aux@backref}[5]{}
\providecommand{\abx@aux@defaultrefcontext}[3]{}
\providecommand{\abx@aux@read@bbl@mdfivesum}[1]{}
\providecommand{\abx@aux@read@bblrerun}{}
\providecommand{\abx@aux@sortscheme}[1]{}
\providecommand{\abx@aux@refsection}[1]{}
\providecommand{\abx@aux@number}[2]{}
\makeatother

\makeatletter
\let\dgm\@undefined
\makeatother

  \newcommand{\concat}{\mathbin{\ast}}

  \newcommand{\id}[1][]{\operatorname{id}_{#1}}
  
  \renewcommand{\hat}{\widehat}

  \newcommand{\ito}{\hookrightarrow}

  \declaretheorem[style=definition,within=section]{definition}
  \declaretheorem[style=definition,numberlike=definition]{example}
  \declaretheorem[style=definition,numberlike=definition]{remark}

  \declaretheorem[style=definition,numberlike=definition]{construction}

  \declaretheorem[style=plain,numberlike=definition]{corollary}
  \declaretheorem[style=plain,numberlike=definition]{lemma}
  \declaretheorem[style=plain,numberlike=definition]{proposition}
  \declaretheorem[style=plain,numberlike=definition]{theorem}

  \declaretheorem[style=plain,numbered=no,name=Theorem]{theorem*}

  \Crefname{corollary}{Corollary}{Corollaries}
  \Crefname{definition}{Definition}{Definitions}
  \Crefname{lemma}{Lemma}{Lemmas}
  \Crefname{proposition}{Proposition}{Propositions}
  \Crefname{remark}{Remark}{Remarks}
  \Crefname{theorem}{Theorem}{Theorems}
  \Crefname{notation}{Notation}{Notations}
  \Crefname{conjecture}{Conjecture}{Conjectures}

  \newlist{axioms}{enumerate}{1}
  \Crefname{axiomsi}{}{}

  \newenvironment{tikzeq*}
  {
    \begingroup
    \begin{equation*}
    \begin{tikzpicture}[baseline=(current bounding box.center)]
  }
  {
    \end{tikzpicture}
    \end{equation*}
    \endgroup
    \ignorespacesafterend
  }

  \tikzset
  {
    diagram/.style=
    {
      matrix of math nodes,
      column sep={4.3em,between origins},
      row sep={4em,between origins},
      text height=1.5ex,
      text depth=.25ex
    },
    over/.style={preaction={draw=white,-,line width=6pt}},
    every to/.style={font=\footnotesize},
    inj/.style={right hook->},
    surj/.style={-{Latex[open]}},
    cof/.style={>->},
    fib/.style={->>},
  }

  \DeclareFontFamily{U}{mathx}{\hyphenchar\font45}

  \DeclareFontShape{U}{mathx}{m}{n}{
    <5> <6> <7> <8> <9> <10>
    <10.95> <12> <14.4> <17.28> <20.74> <24.88>
    mathx10}{}

  \DeclareSymbolFont{mathx}{U}{mathx}{m}{n}

  \DeclareFontFamily{U}{mathb}{\hyphenchar\font45}

  \DeclareFontShape{U}{mathb}{m}{n}{
    <5> <6> <7> <8> <9> <10>
    <10.95> <12> <14.4> <17.28> <20.74> <24.88>
    mathb10}{}

  \DeclareSymbolFont{mathb}{U}{mathb}{m}{n}

  \DeclareMathAccent{\widebar}{0}{mathx}{"73}

  \DeclareMathSymbol{\Rsh}{\mathrel}{mathb}{"E9}

  \DeclareFontFamily{U}{MnSymbolA}{}

  \DeclareFontShape{U}{MnSymbolA}{m}{n}{
    <-6> MnSymbolA5
    <6-7> MnSymbolA6
    <7-8> MnSymbolA7
    <8-9> MnSymbolA8
    <9-10> MnSymbolA9
    <10-12> MnSymbolA10
    <12-> MnSymbolA12}{}

  \DeclareSymbolFont{MnSyA}{U}{MnSymbolA}{m}{n}

  \DeclareMathSymbol{\twoheaddownarrow}{\mathrel}{MnSyA}{27}

  \newcommand{\MSC}[1]{%
    \let\thempfn\relax
    \footnotetext[0]{2020 Mathematics Subject Classification: #1.}
  }

  \newcommand{\dgm}{\operatorname{Dgm}}   % persistence diagram
\newcommand{\cmpl}[1]{\overline{#1}}
\newcommand{\Hom}{\operatorname{Hom}}

\newcommand{\regularpath}[3]{C_{#3}(#1;#2)}
\newcommand{\identity}[1]{\mathrm{id}_{#1}}

\newcommand{\subsets}[1]{\mathcal{P}(#1)}

\newcommand{\allpaths}[3]{K_{#3}(#1;#2)}
\newcommand{\degeneratepaths}[3]{D_{#3}(#1;#2)}
\newcommand{\allowedpaths}[3]{A_{#3}(#1;#2)}

\newcommand{\hPath}[1][]{\operatorname{H}^{\mathrm{Path}}_{#1}}   % path homology, degree as optional arg
\newcommand{\hlgy}[1]{\operatorname{H}^{\mathrm{Path}}(#1)}
\newcommand{\nhlgy}[2]{H_{#1}(#2)}
\newcommand{\nRhlgy}[3]{H_{#1}(#2;#3)}

\newcommand{\catpathcomplex}{\mathcal{P}}

\newcommand{\dbottleneck}[2]{d_{\mathrm{B}}(#1,#2)}

\newcommand{\dhyper}[2]{d_{\mathrm{H}}(#1,#2)}
\newcommand{\dpathcomplex}[2]{d_{\mathrm{path}}(#1,#2)}
\newcommand{\dDhyper}[2]{d_{\mathrm{DH}}(#1,#2)}
\newcommand{\dnetwork}[2]{d_{\mathrm{N}}(#1,#2)}

\newcommand{\dispath}{\operatorname{dis}_{\mathrm{path}}}
\newcommand{\dishyper}{\operatorname{dis}_{\mathrm{H}}}
\newcommand{\disDhyper}{\operatorname{dis}_{\mathrm{DH}}}
\newcommand{\disnetwork}{\operatorname{dis}_{\mathrm{N}}}

\tikzstyle{vertex}=[circle, draw, minimum size=7pt, inner sep=0pt]

\newcommand{\Cyl}{\mathsf{Cyl}}

\DeclareFontFamily{U}{dmjhira}{}
\DeclareFontShape{U}{dmjhira}{m}{n}{ <-> dmjhira }{}

\author{Krzysztof Kapulkin \and Kyle Koyanagi}

\title{Stability of persistent path homology of path complexes}

\date{\today}

\begin{document}

\maketitle

\begin{abstract}
  We show stability of persistent path homology of path complexes.
  As a consequence, we deduce the stability of persistent path homology of hypergraphs and of sequence hypergraphs, and recover the known stability result for digraphs, originally due to Chowdhury and M{\'e}moli.
  
  \noindent \textbf{Keywords:} path homology, persistence, persistent path homology, path complex, hypergraph, network distance, stability, bottleneck distance
\end{abstract}

%  \setlist[enumerate]{label=(\arabic*)}

% Add content here

\tableofcontents*

\section*{Introduction}

Path homology is an invariant of directed graphs, introduced by Grigor'yan, Lin, Muranov, and Yau \cite{grigor2012homologies}.
It is sensitive to the orientation of edges, and in this it differs from invariants that factor through the undirected graph underlying a digraph.
The two digraphs below have the same underlying undirected graph, yet their first path homology groups differ:
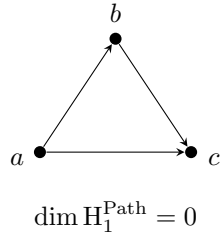
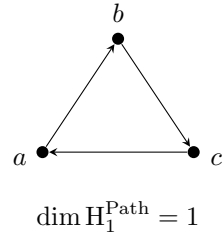
\begin{figure}[H]
\begin{center}
\begin{tikzpicture}[>=stealth, vtx/.style={circle, fill=black, inner sep=1.6pt}]
  \begin{scope}
    \node[vtx] (a) at (0,0) {};      \node at (-0.3,-0.1) {$a$};
    \node[vtx] (b) at (1,1.5) {};    \node at (1,1.85) {$b$};
    \node[vtx] (c) at (2,0) {};      \node at (2.3,-0.1) {$c$};
    \draw[->] (a) -- (b);
    \draw[->] (b) -- (c);
    \draw[->] (a) -- (c);
    \node at (1,-0.85) {$\dim \hPath[1] = 0$};
  \end{scope}
  \begin{scope}[xshift=6cm]
    \node[vtx] (a2) at (0,0) {};     \node at (-0.3,-0.1) {$a$};
    \node[vtx] (b2) at (1,1.5) {};   \node at (1,1.85) {$b$};
    \node[vtx] (c2) at (2,0) {};     \node at (2.3,-0.1) {$c$};
    \draw[->] (a2) -- (b2);
    \draw[->] (b2) -- (c2);
    \draw[->] (c2) -- (a2);
    \node at (1,-0.85) {$\dim \hPath[1] = 1$};
  \end{scope}
\end{tikzpicture}
\end{center}
\end{figure}

Much work has been done on path homology since its introduction, and we make no attempt to survey it here \cite{grigoryans, grigor2020pathcomplexes, grigor2019homology, grigoryan-jimenez-muranov-yau:eilenberg-steenrod, grigoryan-muranov-yau:kunneth, grigoryan-muranov-yau:graphs-simplicial, carranza-doherty-kapulkin-opie-sarazola-wong:cofibration, di-ivanov-mukoseev-zhang:cayley, fu-ivanov:multisquares, hepworth-roff:bigraded}.
Two points from that literature bear directly on what follows.
Path homology is genuinely different from other homology theories of graphs: Barcelo, Greene, Jarrah, and Welker compare it with discrete homology \cite{babson-barcelo-longueville-laubenbacher,barcelo-capraro-white,carranza-kapulkin:cubical-graphs} and show that the two do not agree \cite{barcelo-greene-jarrah-welker:comparison}.
It is closely related via a spectral sequence to magnitude homology \cite{hepworth-willerton}, as shown by Asao \cite{asao}.

A weighted digraph is what is elsewhere called a network, and networks are among the most common data structures in the sciences, arising from the Internet, from social interactions, and from biological systems \cite{newman:complex-networks}.
Recovering their global structure is a central problem in data analysis.
Topological data analysis approaches such problems by extracting features that survive small perturbations of the input, and its main tool is persistent homology \cite{edelsbrunner-letscher-zomorodian:persistence, zomorodian-carlsson:computing-persistent-homology, carlsson:topology-and-data}.
Instead of fixing one scale, one builds a filtration of combinatorial objects indexed by a resolution parameter, applies homology at each stage, and records the scales at which features are born and die.
Since its introduction the method has found applications in areas as disparate as time series analysis \cite{umeda2017time}, neuroscience \cite{gardner-et-al}, and biology \cite{benjamin-et-al}, and this list is far from exhaustive.

For directed data these two lines meet.
Chowdhury and M{\'e}moli extended path homology to the persistent setting, obtaining persistent path homology together with a stability theorem and an algorithm for its computation \cite{chowdhury2018persistent}.
They apply it to a classification task on a database of simulated hippocampal networks, where the question is whether the spike trains of place cells alone reveal the structure of the environment an animal explores.
The same authors developed further tools for asymmetric data, among them a functorial Dowker theorem and a study of network distances and the stability of network invariants \cite{chowdhury-memoli:dowker, chowdhury-memoli:distances}, and path homology has since been applied to temporal networks and to deep feedforward networks \cite{chowdhury-huntsman-yutin:motifs, chowdhury-gebhart-huntsman-yutin:deep-networks}.
Chaplin, Harrington, and Tillmann introduced grounded persistent path homology, a functorial descriptor of edge-weighted digraphs with geometrically interpretable representatives \cite{chaplin-harrington-tillmann:grounded}.
Efficient computation in degree one is due to Dey, Li, and Wang \cite{dey-li-wang:path-homology}.

Path homology is not confined to digraphs.
Grigor'yan, Lin, Muranov, and Yau isolated the notion of a \emph{path complex}, a set of vertices together with a collection of allowed paths closed under deleting the first or the last vertex, and defined path homology at that level of generality \cite{grigor2020pathcomplexes}.
Digraphs and simplicial complexes are both canonically path complexes, and for them path homology recovers digraph path homology and simplicial homology respectively.
Hypergraphs give further examples \cite{grigor2019homology}, as do the sequence hypergraphs of \cref{sec:sequence-hypergraph-stability}.
Path complexes are thus a common home for a range of familiar combinatorial objects, and a theorem proved for path complexes applies to all of them at once.

This is the setting for the present paper, whose subject is stability.
Stability is what makes persistence usable on measured data.
If two objects are close, their persistence diagrams should be close as well, so that a small error in the input cannot produce a large change in the output.
Making this precise calls for a distance on the objects, a distance on diagrams, and a theorem comparing the two.
For point clouds such a theorem is due to Chazal, de Silva, and Oudot \cite{chazal-desilva-oudot:stability-geometric-complexes}, and for directed networks to Chowdhury and M\'emoli \cite{chowdhury2018persistent}.

Our main result, \cref{thm:path-stability}, is a stability theorem for the persistent path homology of filtered path complexes.
To pass from it to the objects above, we set up a covering framework in \cref{sec:framework}, which builds a filtered object from a weight on its generating cells and reduces stability for that object to \cref{thm:path-stability}.
Instantiating the framework gives stability for hypergraphs in \cref{thm:hyp-stability}, for sequence hypergraphs in \cref{thm:dhyp-stability}, and for digraphs in \cref{thm:net-stability}.
The theorem of Chowdhury and M\'emoli for digraphs comes out as a special case.

We claim no novelty in the method.
The proof of \cref{thm:path-stability} follows the template of Chazal, de Silva, and Oudot, adapted to networks by Chowdhury and M\'emoli: a correspondence between two objects is used to build an interleaving of the associated persistence modules, and algebraic stability converts that interleaving into a bound on bottleneck distance.
What we contribute is the adaptation of their metrics to path complexes, hypergraphs, and sequence hypergraphs, together with an arrangement of the argument in which the steps shared by all of these objects are isolated and proved once.

Two other papers call for comment.
Zhang also proves stability of persistent path homology, for weighted digraphs and for edge-weighted path complexes, using the homotopy theory of digraphs together with interleaving and bottleneck distances \cite{zhang:persistent-path-diagrams}.
Those results overlap ours for digraphs.
The difference lies in scope and organization: we prove the path complex case once and obtain hypergraphs, sequence hypergraphs, and digraphs as instances of a single framework, while Zhang works with weighted digraphs and edge-weighted path complexes directly.
Bubenik and Mili\'cevi\'c, in turn, prove a stability theorem of considerable generality, namely that any persistence module arising from a homotopy-invariant functor on filtrations of \v{C}ech closure spaces is stable \cite{bubenik-milicevic}.
Their examples include metric spaces, weighted graphs, and weighted digraphs, subsuming a good deal of existing literature.
Path complexes and hypergraphs, however, are not closure spaces, and path homology is not among the homology theories they construct, so their theorem does not cover the results proved here.
Our argument nonetheless proceeds along the same lines as theirs, and along the same lines as the standard proofs of stability.

The paper is organized as follows.
\Cref{sec:prelims} recalls what we need about path complexes, path homology, and persistence.
\Cref{sec:framework} sets out the covering framework and the four steps by which each subsequent section is organized.
\Cref{sec:path-cpx-stability} proves the stability theorem for path complexes.
\Cref{sec:hypergraph-stability}, \cref{sec:sequence-hypergraph-stability}, and \cref{sec:network-stability} instantiate the framework for hypergraphs, sequence hypergraphs, and digraphs.
\Cref{sec:conclusion} summarizes the results and describes directions for further work.
\section{Preliminaries} \label{sec:prelims}

In this section, we collect the preliminary material on path homology.
Our review is necessarily brief and we refer the reader to the original sources \cite{grigoryans} for a more thorough overview of the subject.
Throughout, $R$ denotes a fixed commutative ring, over which all modules and chain complexes are taken; we suppress it from the notation where no confusion arises.

\subsection{Path Complexes}
For $n \geq 0$, an \emph{$n$-path} on a set $X$ is a sequence $(x_0, \dots, x_n)$ of elements of $X$, also called \emph{vertices}.
In other words, an $n$-path consists of $n+1$ vertices, and the set of all such paths is the iterated product $X^{n+1}$.
The set of all paths on $X$ will be denoted $X^{<\infty}$.
A path $(x_0,\dots,x_n)$ is called \emph{regular} if $x_i \neq x_{i+1}$ for all $0 \leq i \leq n-1$, and \emph{irregular} otherwise.

We define the following $R$-modules:
\begin{itemize}
    \item $\allpaths{X}{R}{n} = R\{X^{n+1}\}$ is the free $R$-module on $X^{n+1}$;
    \item $\degeneratepaths{X}{R}{n} = R \{(x_0, \dots ,x_n) \in X^{n+1} \ | \ x_i = x_{i+1} \text{ for some } 0 \leq i < n \}$ is the $R$-submodule of $\allpaths{X}{R}{n}$ spanned by irregular paths;
    \item $\regularpath{X}{R}{n} = \allpaths{X}{R}{n}/\degeneratepaths{X}{R}{n}$ is the quotient of the $R$-module of all paths by the $R$-submodule of irregular paths.
    (We observe that $\regularpath{X}{R}{n}$ is again free, generated by regular paths.)
\end{itemize}
The \emph{differential} $\partial_n \colon \allpaths{X}{R}{n} \to \allpaths{X}{R}{n-1}$ is a linear operator defined by: 
\[
\partial_n(x_0, \dots, x_n) = \sum_{i=0}^{n}(-1)^i(x_0, \dots, \hat{x_i}, \dots, x_n)
\]
where $\hat{x_i}$ means we remove $x_i$ from the path.
As $\partial_n(\degeneratepaths{X}{R}{n}) \subseteq \degeneratepaths{X}{R}{n-1}$, the differential descends to $\partial_n \colon \regularpath{X}{R}{n} \to \regularpath{X}{R}{n-1}$. 
It is easy to verify that both $\allpaths{X}{R}{\bullet}$ and $\regularpath{X}{R}{\bullet}$ with this boundary operator are chain complexes.
    
A set function $f \colon X \to Y$ induces  a morphism of chain complexes $f_* \colon \allpaths{X}{R}{n} \to \allpaths{Y}{R}{n}$ by applying $f$ to each basis element componentwise.
Furthermore, since $f_*(\degeneratepaths{X}{R}{n}) \subseteq \degeneratepaths{Y}{R}{n}$, the morphism $f_*$ descends to a morphism of chain complexes of regular paths $f_* \colon \regularpath{X}{R}{*} \to \regularpath{Y}{R}{*}$.

\begin{definition} \label{def:path-complex}
A \emph{path complex} $P$ on a vertex set $X$ is a collection of elementary paths such that:
\begin{enumerate}
    \item $(x) \in P$ for every vertex $x$;
    \item if $(x_0,\dots,x_n) \in P$, then $(x_0,\dots, x_{n-1})\in P$ and $(x_1, \dots, x_n) \in P$.
\end{enumerate}
\end{definition}

Given a path complex $P$, we denote its paths of length $n \geq 0$ by $P_n$.
The elements of $P$ are called \emph{allowed elementary paths}, and the elementary paths that are not in $P$ are called \emph{non-allowed}.
With this notation in place, we can suppress $X$ from notation altogether since the first condition implies that $X = P_0$.

The first examples of path complexes are given by digraphs:

\begin{example}
Recall that a \emph{directed graph} (or \emph{digraph}) is a pair $G=(X, E_G)$, where $X$ is a set (of \emph{vertices}) and $E_G \subseteq X \times X$ is a reflexive relation on $X$, i.e., a set of edges containing all pairs of the form $(x, x)$.
Given a digraph $G$, we can form a path complex by declaring a path $(x_0, x_1, \ldots, x_n)$ to be in $P$ if $(x_i, x_{i+1})$ is in $E_G$ for all $0 \leq i < n$.
\end{example}

\begin{definition}
A \emph{map of path complexes} $f \colon P \to Q$ is a function $f \colon P_0 \to Q_0$ such that for any path $(x_0, \dots, x_n) \in P$, the path $(f(x_0), \dots, f(x_n))$ lies in $Q$.
\end{definition}

The identity map $\id[P] \colon P \to P$ is a map of path complexes and that maps of path complexes compose, thus giving us a well-defined category $\catpathcomplex$ whose objects are path complexes and whose morphisms are morphisms of path complexes. 

\begin{definition}
For any integer $n \geq 0$, define the $R$-module $\allowedpaths{P}{R}{n}$ of \emph{allowed $n$-paths} spanned by the elementary $n$-paths of $P$:
\[
\allowedpaths{P}{R}{n} = R\{P_n\}.
\] 
\end{definition}

%Note that the collection of all paths on a vertex set $X$ is itself a path complex, the \emph{full path complex} on $X$; for it we have $\allowedpaths{P}{R}{n} = \allpaths{X}{R}{n}$.

We would like to restrict the boundary operator $\partial_n$ to allowed paths and obtain a chain complex, but this need not be possible: for $p \in \allowedpaths{P}{R}{n}$ we might not have that $\partial_n(p) \in \allowedpaths{P}{R}{n-1}$.
To see this, consider the path complex $P$ coming from a digraph with: 
\[
    P_0 = \{0,1,2\}, \qquad P_1 = \{(0,1),(1,2)\}, \qquad P_2 = \{(0,1,2)\}. 
\]
We calculate $\partial_2(0,1,2) = (1,2) - (0,2) + (0,1)$.
Since $(0,2) \notin P_1$, the path $p = (0,1,2)$ satisfies $\partial_2 p \notin \allowedpaths{P}{R}{1}$.
To circumvent this issue we restrict to further submodules of $\allowedpaths{P}{R}{n}$.

\begin{definition}
For a path complex $P$, define the complex of \emph{$\partial$-invariant paths} by:
\[
    \Omega_n(P; R) = \{v \in \allowedpaths{P}{R}{n} \ | \ \partial_n(v) \in \allowedpaths{P}{R}{n-1}\}\text{.}
\]
\end{definition}

It is straightforward to verify that 
\[
\cdots \longrightarrow \Omega_{n+1}(P; R) \longrightarrow \Omega_n (P; R) \longrightarrow \Omega_{n-1}(P; R) \longrightarrow \cdots \longrightarrow \Omega_0(P; R) \longrightarrow 0
\]
is a chain complex.
Moreover, any morphism $f \colon P \to Q$ of path complexes induces a morphism of chain complexes $f_* \colon \Omega_*(P; R) \to \Omega_*(Q; R)$, making $\Omega_*$ a functor from the category of path complexes to the category of chain complexes.

To see that this restriction yields interesting results, consider the following example.

\begin{example}
Consider the digraph 
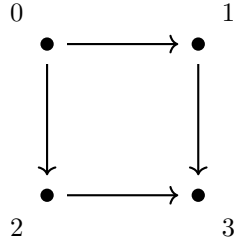
\begin{center}
\begin{tikzpicture}[
    vertex/.style = {circle, fill=black, minimum size = 5pt, outer sep = 5 pt, inner sep = 0pt}
]
\node (A) at (0,0) [vertex, label=above left:$0$] {};
\node (B) at (2,0) [vertex, label=above right:$1$] {};
\node (C) at (0,-2) [vertex, label=below left:$2$] {};
\node (D) at (2,-2) [vertex, label=below right:$3$] {};

\draw[->,thick] (A) -- (B);
\draw[->,thick] (A) -- (C);
\draw[->,thick] (B) -- (D);
\draw[->,thick] (C) -- (D);
\end{tikzpicture}
\end{center}
Let $P$ be the corresponding path complex.
By definition, the allowed $2$-paths of $P$ are $(0, 1, 3)$ and $(0, 2, 3)$, so $\allowedpaths{P}{R}{2}$ is spanned by these two paths.
Their boundaries are $\partial_2(0, 1, 3) = (1, 3) - (0, 3) + (0, 1)$ and $\partial_2(0, 2, 3) = (2, 3) - (0, 3) + (0, 2)$.
Since $(0, 3)$ is not an edge in the graph, we have $(0, 3) \notin \allowedpaths{P}{R}{1}$, but the linear combination $(0, 1, 3) - (0, 2, 3)$ has boundary $\partial_2((0, 1, 3) - (0, 2, 3)) = (1, 3) + (0, 1) - (2, 3) - (0, 2)$, which is allowed. 
Hence $\Omega_2(P;R)$ is spanned by $(0, 1, 3) - (0, 2, 3)$.
\end{example}

\begin{definition}
    Let $P$ be a path complex. 
    The \emph{path homology} of $P$ with $R$-coefficients, denoted $\nRhlgy{\bullet}{P}{R}$, is the homology of the chain complex $(\Omega_\bullet(P;R),\partial_\bullet)$; that is, the $n$-th path homology group is given by
    \[
        \nhlgy{n}{P; R} = \ker(\partial_n)/\operatorname{im}(\partial_{n+1}).
    \]
\end{definition}

Using functoriality of homology, we obtain that a map of path complexes $f \colon P \to Q$ induces a homomorphism of graded abelian groups $f_* \colon \nhlgy{*}{P; R} \to \nhlgy{*}{Q; R}$ on homology.

\subsection{Homotopy in the category of path complexes}

In classical algebraic topology, homology is homotopy invariant, i.e., homotopic maps induce the same homomorphism in homology, and consequently, homotopy equivalent spaces have isomorphic homology groups.
The same is true for path homology of path complexes, but to state this result, we must first establish the notion of homotopy of maps of path complexes.
To do so, we recall the definition of the box product of path complexes, which is due to Grigor'yan, Lin, Muranov, and Yau \cite{grigor2020pathcomplexes}, where it is used to establish a K\"unneth formula for path homology.

\begin{definition}
    The box product of path complexes $P$ and $Q$ is a path complex $P \square Q$ on the set of vertices $P_0 \times Q_0$ and whose $n$-paths are sequences
    \[ ((x_0, y_0), (x_1, y_1), \ldots, (x_n, y_n)),\]
    such that: 
    \begin{enumerate}
        \item for each $0 \leq i < n$, exactly one of $x_i = x_{i+1}$ and $y_i = y_{i+1}$ holds; and
        \item the sequences obtained from $(x_0, x_1, \ldots, x_n)$ and $(y_0, y_1, \ldots, y_n)$ by collapsing each run of repeated consecutive vertices to a single vertex are allowed paths in $P$ and $Q$ respectively.
    \end{enumerate}
\end{definition}

Condition (1) forces consecutive vertices to differ, so every path of $P \square Q$ is regular.
This matches the original source, where step-like paths are regular by definition and the regular boundary operator is used throughout, and it costs nothing here, since irregular paths are already quotiented out in $\regularpath{X}{R}{\bullet}$.
The collapsing in (2) is needed because each step leaves one of the two coordinates fixed, so that as soon as $n \geq 1$ at least one of the unreduced sequences $(x_0, \ldots, x_n)$ and $(y_0, \ldots, y_n)$ is irregular.

Note that in the original source \cite{grigor2020pathcomplexes}, this notion is referred to as the \emph{Cartesian} product; we chose to diverge from this naming convention since it could easily lead to confusion with the categorical product.

Consider the path complex $I$ arising from the following digraph:
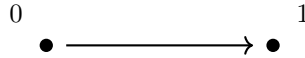
\begin{center}
\begin{tikzpicture}[
    vertex/.style = {circle, fill=black, minimum size = 5pt, outer sep = 5 pt, inner sep = 0pt}
]
\node (A) at (-1.5,0) [vertex, label=above left:$0$] {};
\node (B) at (1.5,0) [vertex, label=above right:$1$] {};

\draw[->,thick] (A) -- (B);
\end{tikzpicture}
\end{center}

For a path complex $P$, the \emph{cylinder} on $P$ is the path complex $\Cyl(P) = P \square I$.
Concretely, a path of $\Cyl(P)$ is a staircase whose $I$-coordinate is constant apart from a single step from $0$ to $1$; its $I$-projection collapses to the edge $(0,1)$, and its $P$-projection collapses to a path of $P$.
There are canonical inclusions $i_0, i_1 \colon P \hookrightarrow \Cyl(P)$ induced by the set maps:
\[ i_0(x) = (x, 0) \qquad \text{and} \qquad i_1(x) = (x, 1)\text{.} \]

\begin{definition}\leavevmode
\begin{itemize}
    \item For maps $f, g \colon P \to Q$ of path complexes, a \emph{one-step homotopy} from $f$ to $g$ is a map of path complexes $\alpha \colon \Cyl(P) \to Q$ such that the following diagram commutes:
    \[
    \begin{tikzcd}[column sep = huge]
        P \arrow[d, "i_0", swap] \arrow[dr, "f", bend left = 25] & \\
        \Cyl(P) \arrow[r, "\alpha"] & Q \\
        P \arrow[u, "i_1"] \arrow[ur, "g", swap, bend right = 25]&
    \end{tikzcd}
    \]
    \item A \emph{homotopy} between maps $f, g \colon P \to Q$ of path complexes is a sequence of maps $(f = f_0, f_1, \ldots, f_n = g)$ such that for all $0 \leq i < n$, we have that either $(f_i, f_{i+1})$ or $(f_{i+1}, f_i)$ are one-step homotopic.
    We write $f \simeq g$ to indicate that $f$ and $g$ are homotopic, i.e., that there exists a homotopy from $f$ to $g$.
    \item A map $f \colon P \to Q$ of path complexes is a homotopy equivalence if there exists a map $g \colon Q \to P$ of path complexes along with homotopies $gf \simeq \id[P]$ and $fg \simeq \id[Q]$.
\end{itemize}
\end{definition}

With these definitions in place, we can now state the homotopy invariance property of path homology.

\begin{theorem}[\cite{grigoryans}] \label{thm:homotopic-equal-on-hlgy}
    Homotopic maps of path complexes $f, g \colon P \to Q$ induce chain homotopic maps $f_*, g_* \colon \Omega_*(P; R) \to \Omega_*(Q; R)$, and consequently equal maps on homology groups. \qed
\end{theorem}

\subsection{Persistence, interleavings, and stability} \label{subsec:persistence-background}

We briefly recall the persistence-theoretic notions used throughout; for thorough treatments see \cite{chazal-desilva-glisse-oudot:structure-stability, otter-et-al:roadmap}.
Fix a field $\mathbb{F}$ and a homological degree $k \in \mathbb{N}$.

A \emph{persistence module} is a family $\{V_\delta\}_{\delta \geq 0}$ of $\mathbb{F}$-vector spaces together with linear maps $v_\delta^{\delta'} \colon V_\delta \to V_{\delta'}$ for $\delta \leq \delta'$ satisfying $v_\delta^\delta = \id$ and $v_{\delta'}^{\delta''} \circ v_\delta^{\delta'} = v_\delta^{\delta''}$.
Applying degree-$k$ path homology with coefficients in $\mathbb{F}$ to a filtered path complex $P = \{P^\delta\}_{\delta \geq 0}$ yields a persistence module $\hPath[k](P) = \{\hPath[k](P^\delta)\}_{\delta \geq 0}$, with structure maps induced by the filtration inclusions.

For $\varepsilon \geq 0$, an \emph{$\varepsilon$-interleaving} between persistence modules $V$ and $W$ is a pair of families of maps $\varphi_\delta \colon V_\delta \to W_{\delta + \varepsilon}$ and $\psi_\delta \colon W_\delta \to V_{\delta + \varepsilon}$ commuting with the structure maps and satisfying $\psi_{\delta + \varepsilon} \circ \varphi_\delta = v_\delta^{\delta + 2\varepsilon}$ and $\varphi_{\delta + \varepsilon} \circ \psi_\delta = w_\delta^{\delta + 2\varepsilon}$.
The \emph{interleaving distance} is
\[
    d_{I}(V, W) = \inf \{ \varepsilon \geq 0 \mid V \text{ and } W \text{ are } \varepsilon\text{-interleaved} \},
\]
with the convention that $d_I(V,W) = \infty$ when no interleaving exists.

A persistence module that is pointwise finite-dimensional decomposes as a direct sum of interval modules \cite{crawley-boevey:decomposition}, and the multiset of the corresponding intervals is its \emph{persistence diagram} (or \emph{barcode}), which we denote $\dgm_k(P)$ for the degree-$k$ path homology of $P$; each interval is drawn as the point in the extended plane given by its endpoints.
The \emph{bottleneck distance} $\dbottleneck{\dgm_k(P)}{\dgm_k(Q)}$ between two diagrams is the infimum, over all partial matchings of their points, of the largest $L^\infty$ distance a matching moves a point (points may be matched to the diagonal).
The two distances are related by the isometry theorem: for pointwise finite-dimensional modules, the bottleneck distance between the diagrams equals the interleaving distance between the modules \cite{chazal-desilva-glisse-oudot:structure-stability}.
The stability results of this paper are upper bounds on the bottleneck distance between the persistence diagrams of two objects in terms of a distance between the objects themselves.
\section{The Framework} \label{sec:framework}

All the inputs to our constructions will be weighted objects: weighted path complexes, weighted (sequence) hypergraphs, and weighted digraphs.
A \emph{weight} on a set $A$ is a function $w \colon A \to \mathbb{R}_{\geq 0} \cup \{\infty\}$.
For us, the set $A$ will be given by paths in a path complex, edges in a hypergraph or a digraph, or ordered edges in a sequence hypergraph.

The stability results in the following sections all rest on a single construction that extends a weight defined on the ``generating cells'' of a combinatorial object (paths, edges, ordered edges) to all cells by covering, and the object is then filtered by sublevel sets of the extended weight.
We isolate this construction here, in a form general enough to be instantiated in each subsequent section.
The key structural fact, that the extension is idempotent, is proven once in \cref{lem:covering-idempotent} and used in every section that follows.

The construction is an abstraction of a familiar one from the theory of directed graphs.
Given a weighted digraph, the shortest-path distance between two vertices $x$ and $y$ is the least total weight of a directed path from $x$ to $y$.
Two features of this make it work.
First, a single edge is a path from its source to its target, so the distance from $x$ to $y$ is at most the weight of a direct edge, i.e., the distance never exceeds the raw weight.
Second, a path from $x$ to $y$ may be assembled from a path $x \to z$ followed by a path $z \to y$, and the shortest-path distance already accounts for every such decomposition, so recomputing shortest paths from the shortest-path distances changes nothing; in other words, the operation is idempotent.

\begin{definition} \label{def:covering-relation}
    Let $A$ be a set, whose elements we call \emph{cells}.
    A \emph{covering relation} on $A$ assigns to each cell $a \in A$ a collection of finite subsets $S \subseteq A$, called \emph{covers} of $a$, subject to two axioms:
    \begin{enumerate}
        \item \label{ax:unit} \textup{(Unit)} the singleton $\{a\}$ covers $a$;
        \item \label{ax:composition} \textup{(Composition)} if $S$ covers $a$ and, for each $s \in S$, the set $T_s$ covers $s$, then $\bigcup_{s \in S} T_s$ covers $a$.
    \end{enumerate}
\end{definition}

In our motivating example of a (weighted) digraph, the cells are pairs of vertices, and a set of edges covers $(x,y)$ when it forms a directed path from $x$ to $y$: the unit axiom says that a single edge from $x$ to $y$ is (trivially) a directed path, and the composition axiom says that the concatenation of directed paths is again a directed path.
In each application, ``$S$ covers $a$'' unfolds to a concrete connectivity-and-containment condition: path connectedness for path complexes (\cref{sec:path-cpx-stability}), edge connectedness for hypergraphs (\cref{sec:hypergraph-stability}), sequence containment for sequence hypergraphs (\cref{sec:sequence-hypergraph-stability}), and directed paths for digraphs (\cref{sec:network-stability}).
In every case a cover of $a$ is a family out of which $a$ can be reassembled, and not merely a family whose members mention the constituents of $a$; it is this that makes the composition axiom hold.
Moreover, in each case the content of verifying that it is a covering relation is precisely the composition axiom.

\begin{construction} \label{con:induced-weight}
    Let $w$ be a weight on a set $A$ with a covering relation.
    The \emph{induced weight} $\overline{w} \colon A \to \mathbb{R}_{\geq 0} \cup \{\infty\}$ is
    \[
        \overline{w}(a) = \inf\Big\{ \textstyle\sum_{s \in S} w(s) \;\Big|\; S \text{ covers } a \Big\},
        \]
    with the convention that $\inf\varnothing = \infty$.
\end{construction}

\begin{lemma} \label{lem:covering-idempotent}
    Let $w$ be a weight on a set $A$ with a covering relation, and let $\overline{w}$ be the induced weight. Then:
    \begin{enumerate}
        \item \label{item:cov-decreases} $\overline{w}(a) \leq w(a)$ for every $a \in A$;
        \item \label{item:cov-subadditive} $\overline{w}$ is covering-subadditive: $\overline{w}(a) \leq \sum_{s \in S} \overline{w}(s)$ whenever $S$ covers $a$;
        \item \label{item:cov-idempotent} $\overline{(\cdot)}$ is idempotent: $\overline{\overline{w}} = \overline{w}$.
    \end{enumerate}
    Moreover, $\overline{w}$ is the largest covering-subadditive weight below $w$.
\end{lemma}

\begin{proof}
    \Cref{item:cov-decreases} follows from the unit axiom: the singleton $\{a\}$ covers $a$, so the infimum defining $\overline{w}(a)$ is over a family containing the value $w(a)$.

    For \cref{item:cov-subadditive}, let $S$ cover $a$ and fix $\varepsilon > 0$; we may assume $\sum_{s \in S} \overline{w}(s) < \infty$.
    For each $s \in S$, choose a cover $T_s$ of $s$ with $\sum_{t \in T_s} w(t) \leq \overline{w}(s) + \varepsilon/|S|$.
    By the composition axiom \cref{ax:composition}, $T = \bigcup_{s \in S} T_s$ covers $a$, so
    \[
        \overline{w}(a) \leq \sum_{t \in T} w(t) \leq \sum_{s \in S} \sum_{t \in T_s} w(t) \leq \sum_{s \in S} \overline{w}(s) + \varepsilon.
    \]
    As $\varepsilon > 0$ was arbitrary, $\overline{w}(a) \leq \sum_{s \in S} \overline{w}(s)$.

    For \cref{item:cov-idempotent}, apply \cref{item:cov-decreases} to $\overline{w}$ to get $\overline{\overline{w}} \leq \overline{w}$.
    For the reverse inequality, let $S$ cover $a$.
    By \cref{item:cov-subadditive}, $\overline{w}(a) \leq \sum_{s \in S} \overline{w}(s)$; taking the infimum over all covers $S$ of $a$ gives $\overline{w}(a) \leq \overline{\overline{w}}(a)$.

    Finally, \cref{item:cov-subadditive,item:cov-decreases} show $\overline{w}$ is covering-subadditive and below $w$.
    If $v$ is any covering-subadditive weight with $v \leq w$, then for every cover $S$ of $a$ we have $v(a) \leq \sum_{s \in S} v(s) \leq \sum_{s \in S} w(s)$; taking the infimum gives $v(a) \leq \overline{w}(a)$.
    Thus $\overline{w}$ is the largest such weight.
\end{proof}

Every instance in this paper uses \cref{lem:covering-idempotent} in the same way: the induced weight is formed, and its sublevel sets are taken as a filtration.

\begin{construction} \label{con:generated-filtration}
    Let $A$ be a set equipped with a covering relation and let $w$ be a weight on $A$.
    For $\delta \geq 0$ put
    \[
        A^\delta = \{\, a \in A \mid \overline{w}(a) \leq \delta \,\}.
    \]
    The family $\{A^\delta\}_{\delta \geq 0}$ is increasing, and we call it the \emph{filtration generated by} $w$.
    Its entry weight is the induced weight itself:
    \[
        l(a) = \inf\{\, \delta \geq 0 \mid a \in A^\delta \,\} = \overline{w}(a).
    \]
    By \cref{lem:covering-idempotent} we have $\overline{w} \leq w$, and $\overline{w}$ is unchanged by repeating the construction.
\end{construction}

Each of the following sections applies \cref{con:generated-filtration} to its own covering relation; the only work left in each case is to check that the sublevel sets are again objects of the expected kind.

In each of those sections the set $A$ carries additional structure, and the weighted object is a pair consisting of that structure together with a weight on its cells.
For such an object $G$ with weight $w_G$, we write $\cmpl{G}$ for the same underlying structure equipped instead with the induced weight $\overline{w}_{G}$, and call $\cmpl{G}$ the \emph{completion} of $G$.
By \cref{item:cov-idempotent} of \cref{lem:covering-idempotent} the induced weight is unchanged by repeating the construction, so completion is idempotent: $\cmpl{\cmpl{G}} = \cmpl{G}$.

\subsection{Relation to existing notions} \label{subsec:coverage-literature}

Each half of this framework, the covering relation and the induced weight, is an established structure in the existing literature under a different name, although we have not found the two combined, or applied to stability proofs elsewhere.

A \emph{Grothendieck pretopology} is a way of defining covers in an arbitrary category; it assigns to every object $a$ a collection of covers, i.e., sets of maps $\{x_i \to a\}_i$, subject to axioms including that an object covers itself (our unit axiom), that covers compose (our composition axiom), and that the pullback of a cover along any map is again a cover \cite[C2.1]{johnstone2002elephant}.
Our covering relations satisfy analogues of the first two but not the third: our cells (paths and edges) have no maps to pull back along, so the pullback-stability axiom has no counterpart here.

The induced weight is likewise a familiar object.
By \cref{lem:covering-idempotent}, $\overline{w}(a)$ is the cheapest total cost of a cover of $a$.
On a digraph (\cref{sec:network-stability}), where a cover of $(a,b)$ is a directed path from $a$ to $b$, this is exactly the shortest-path distance, so $w \mapsto \overline{w}$ is the all-pairs shortest-path computation.
In this generality it is known as the \emph{algebraic path problem}, and it is carried out in the \emph{tropical}, or \emph{min-plus}, semiring $\mathbb{R}_{\geq 0} \cup \{\infty\}$, with $\inf$ for addition and $+$ for multiplication \cite{baccelli1992synchronization}.
That semiring is the setting for the generic shortest-distance algorithms of \cite{mohri2002semiring}, including Dijkstra's.
Idempotency (\cref{item:cov-idempotent}) says this computation is a closure operator: one pass already reaches the fixed point.
In the view of metric spaces as enriched categories \cite{lawvere1973metric}, $\overline{w}$ is the distance freely generated by $w$.

Two further distinctions are worth drawing.
Our covers are covers of a single cell by generating cells, not of a whole space by open sets as in the nerve-based \v{C}ech and Dowker constructions of topological data analysis \cite{bauer-kerber-roll-rolle:nerve, bauer-edelsbrunner-cech-delaunay, chowdhury-memoli:dowker}.
And our weighted objects are not the weighted simplicial complexes of \cite{dawson1990weighted, ren2018weighted}, where the weight is built into a modified boundary map; here the boundary map is the usual one, and the weight only records when each cell enters the filtration.

\subsection{Framework for stability proofs} \label{subsec:common-pattern}

\Cref{sec:path-cpx-stability} proves the key result of the paper, stability of the persistent path homology of path complexes.
\Cref{sec:hypergraph-stability,sec:sequence-hypergraph-stability,sec:network-stability} then deduce the remaining stability results from it, each by instantiating the framework above together with the results of \cref{sec:path-cpx-stability}.
To aid readability, each of these sections follows the same organizational framework:
\begin{description}
    \item[Step 1: The covering relation.] 
    Identify the covering relation on the object's cells, and verify the two axioms; the induced weight and its idempotency then come from \cref{lem:covering-idempotent}. Sublevel sets of the induced weight give a filtered object.
    \item[Step 2: The associated path complex.] 
    Define the path complex associated to the combinatorial object in question.
    \item[Step 3: The induced path weight.] 
    Compare the weight that the associated path complex already carries with the one induced on it by the path-complex covering relation.
    The two agree for digraphs, and for (sequence) hypergraphs when $q \leq 2$, in which case the weight is a fixed point in the sense of \cref{lem:covering-idempotent}.
    This step relates the two ways of passing from an object to a weighted path complex; neither Step 4 nor the stability theorem depends on it.
    \item[Step 4: Distance comparison.] Bound the path complex distance $\dpathcomplex{\,\cdot\,}{\cdot}$ by the object's own distance.
    Stability then follows by feeding Step 4 into \cref{thm:path-stability}.
\end{description}

The distance comparisons of Step 4 also share an elementary ingredient, which we isolate here.

\begin{lemma} \label{lem:sup-comparison}
    Let $I$ be a finite set, let $(a_i)_{i \in I}$ and $(b_i)_{i \in I}$ be families of real numbers, and let $c \geq 0$.
    If $|a_i - b_i| \leq c$ for every $i \in I$, then
    \[
        \Big| \sup_{i \in I} a_i - \sup_{i \in I} b_i \Big| \leq c.
    \]
\end{lemma}
\begin{proof}
    Since $I$ is finite, both suprema are attained.
    Choose $i_0$ with $\sup_i a_i = a_{i_0}$; then $\sup_i a_i = a_{i_0} \leq b_{i_0} + c \leq \sup_i b_i + c$.
    Exchanging the roles of the two families gives $\sup_i b_i \leq \sup_i a_i + c$, and the two inequalities together give the claim.
\end{proof}

\Cref{tab:instances} summarizes how the four instances fill in this framework.
Path complexes (\cref{sec:path-cpx-stability}) are the base case: there the object already \emph{is} a path complex, so Step 2 is vacuous and the covering relation acts directly on paths.
In the remaining three sections the object is a different combinatorial structure, and what changes from one to the next is entirely the first two rows, the cells and the covering relation on them, while the surrounding machinery is identical.
The contrast between hypergraphs and sequence hypergraphs is the sharpest: the cells pass from unordered to ordered subsets, and correspondingly the covering relation from order-blind edge connectedness to order-sensitive sequence containment; everything else is unchanged.

\begin{table}[H]
\renewcommand{\arraystretch}{1.4}
\begin{center}
\begin{tabular}{@{}p{0.18\textwidth}p{0.19\textwidth}p{0.24\textwidth}p{0.19\textwidth}@{}}
\toprule
& \textbf{Path complexes} & \textbf{Hypergraphs} & \textbf{Digraphs} \\
& (\cref{sec:path-cpx-stability}) & (\cref{sec:hypergraph-stability}, \cref{sec:sequence-hypergraph-stability}) & (\cref{sec:network-stability}) \\
\midrule
Cells & elementary paths & (ordered) subsets of $X$ & vertex pairs $(a,b)$ \\
Covering relation & path connectedness & edge connectedness, and sequence containment & directed paths \\
Path complex & itself & density-$q$ complex $P^q(G)$ & complex $P(G)$ \\
Object distance & $\dpathcomplex{\,\cdot\,}{\cdot}$ & $\dhyper{\cdot}{\cdot}$ / $\dDhyper{\cdot}{\cdot}$ & $\dnetwork{\cdot}{\cdot}$ \\
\bottomrule
\end{tabular}
\end{center}
\caption{The four instances of the covering framework.}
\label{tab:instances}
\end{table}
\section{Stability of Persistent Path Homology of Path Complexes} \label{sec:path-cpx-stability}
In this section we prove stability of persistent path homology of path complexes, following the general stability strategy of \cite{bubenik-milicevic, chazal-desilva-oudot:stability-geometric-complexes}.
This is the base case of the paper: the later sections reduce to it by presenting their objects as path complexes.
It is also the first instance of the covering framework of \cref{sec:framework}, with path connectedness playing the role of the covering relation.
We introduce filtered path complexes, their entry weights, and the distance between them, and then prove the stability theorem.

\subsection{Filtered path complexes}

\begin{definition} \label{def:filtered-path-complex}
    A \emph{filtered path complex} is a family $P = \{P^\delta\}_{\delta \geq 0}$ of path complexes on a fixed vertex set, together with inclusions $P^\delta \ito P^{\delta'}$ for all $\delta \leq \delta'$, such that every length-$0$ path lies in $P^0$.
    We write $P_0$ for the common vertex set, and $P_0^{<\infty}$ for the set of all elementary paths on $P_0$.
\end{definition}

A filtered path complex is equivalently described by a single function recording when each path appears.

\begin{definition} \label{def:entry-weight}
    Let $P$ be a filtered path complex.
    The \emph{entry weight} of $P$ is the function $l_P \colon P_0^{<\infty} \to \mathbb{R}_{\geq 0} \cup \{\infty\}$ given by
    \[
        l_P(p) = \inf\{\, \delta \geq 0 \mid p \in P^\delta \,\},
    \]
    the earliest stage of the filtration containing $p$, with the convention that $l_P(p) = \infty$ if $p \notin P^\delta$ for any $\delta$.
\end{definition}

Path complexes are the first instance of the covering framework of \cref{sec:framework}: the cells are the elementary paths, and the covering relation is \emph{path connectedness}, which specializes \cref{def:covering-relation} to this setting.
%The entry weight will turn out to be subadditive with respect to this relation, and it is that subadditivity on which the stability argument rests.

Throughout, a \emph{subpath} of $p = (x_0, \dots, x_n)$ means a contiguous one, that is, a path $(x_i, x_{i+1}, \dots, x_j)$ with $0 \leq i \leq j \leq n$.
These are exactly the paths that \cref{def:path-complex} forces to be allowed whenever $p$ is.
A cover of $p$ will be a family of paths out of which $p$ can be assembled by concatenation, consecutive pieces being glued along a shared stretch of vertices.
For paths $s = (y_0, \dots, y_j)$ and $t = (z_0, \dots, z_\ell)$ whose last and first $k \geq 1$ vertices agree, i.e.\ $(y_{j-k+1}, \dots, y_j) = (z_0, \dots, z_{k-1})$, we write
\[
    s \concat t = (y_0, \dots, y_j, z_k, \dots, z_\ell)
\]
for their concatenation along that overlap, and $s_1 \concat \cdots \concat s_n$ for the iterated concatenation along a specified choice of overlaps.

\begin{definition} \label{def:path-connected}
    A set $S \subseteq P_0^{<\infty}$ \emph{covers} a path $p$ if there are elements $S_1, \dots, S_n \in S$, not necessarily distinct, and subpaths $s_i$ of $S_i$, such that
    \[
        p = s_1 \concat \cdots \concat s_n
    \]
    for some choice of overlaps.
    Such an $S$ is called a \emph{path connected cover} of $p$.
\end{definition}

Everything we need about this relation follows from a single observation about how subpaths interact with concatenation.

\begin{lemma} \label{lem:window-of-concatenation}
    Let $p = s_1 \concat \cdots \concat s_n$.
    \begin{enumerate}
        \item \label{item:window-decomposes} If $u$ is a subpath of $p$, then $u = t_1 \concat \cdots \concat t_r$ for some subpaths $t_j$ of the $s_i$.
        \item \label{item:window-short} If $u$ is a subpath of $p$ on at most two vertices, then $u$ is a subpath of a single $s_i$.
    \end{enumerate}
\end{lemma}

\begin{proof}
    Index the vertices of $p$ by positions $0, 1, \dots, N$, so that each $s_i$ is the subpath of $p$ on a block of consecutive positions $B_i = [a_i, b_i]$, where $a_1 = 0$, $b_n = N$, and $a_{i+1} \leq b_i$ for every $i < n$, the last inequality expressing that consecutive pieces share at least one vertex.
    Let $u$ occupy the positions $[c,d]$.

    For \cref{item:window-decomposes}, let $i_1 < \cdots < i_r$ enumerate those $i$ with $B_i \cap [c,d] \neq \varnothing$; these indices are consecutive, since the $B_i$ are consecutive blocks covering $[0,N]$.
    Put $t_j$ equal to the subpath of $p$ on $B_{i_j} \cap [c,d]$, which is a subpath of $s_{i_j}$.
    For consecutive $j$ the sets $B_{i_j} \cap [c,d]$ and $B_{i_j + 1} \cap [c,d]$ meet, because
    \[
        \max(a_{i_j+1},\, c) \leq \min(b_{i_j},\, d)
    \]
    follows from the four inequalities $a_{i_j+1} \leq b_{i_j}$, $a_{i_j+1} \leq d$, $c \leq b_{i_j}$ and $c \leq d$.
    Hence $u = t_1 \concat \cdots \concat t_r$ along these overlaps.

    For \cref{item:window-short} we have $d \leq c+1$.
    Let $j$ be the largest index with $a_j \leq c$, which exists as $a_1 = 0$.
    Then $c \leq b_j$: this is clear when $j = n$ because $b_n = N \geq c$, and otherwise $b_j < c$ would give $a_{j+1} \leq b_j < c$, contradicting maximality.
    If $b_j \geq d$ then $[c,d] \subseteq B_j$ and we are done.
    Otherwise $b_j = c < d$, so $a_{j+1} \leq b_j = c$, contradicting maximality unless $j = n$; and if $j = n$ then $b_n = N \geq d$.
\end{proof}

For instance, in the path complex $P$ induced by the digraph
    
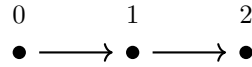
\begin{center}
    \begin{tikzpicture}[
        vertex/.style = {circle, fill=black, minimum size = 5pt, outer sep = 5 pt, inner sep = 0pt}
        ]
        \node (A) at (-1.5,0) [vertex, label=above:$0$] {};
        \node (B) at (0,0) [vertex, label=above:$1$] {};
        \node (C) at (1.5,0) [vertex, label=above:$2$] {};

        \draw[->,thick] (A) -- (B);
        \draw[->,thick] (B) -- (C);
    \end{tikzpicture}
    \end{center}
the subset $S = \{(0,1),(1,2)\}$ covers the path $(0,1,2)$, since $(0,1,2) = (0,1) \concat (1,2)$.
It does not cover $(0,2)$: no subpath of $(0,1)$ or of $(1,2)$ has $0$ and $2$ adjacent, and by \cref{item:window-short} of \cref{lem:window-of-concatenation} a two-vertex path is never assembled from more than one piece.
Nor does $\{(2,1,0)\}$ cover $(0,1,2)$.
Both failures are as they should be: covering sees the adjacencies and the order of $p$, and not merely its vertices.

\begin{lemma} \label{lem:path-covering-relation}
    \Cref{def:path-connected} is a covering relation on $P_0^{<\infty}$ in the sense of \cref{def:covering-relation}.
    Moreover, if $S$ covers $p$, then $S$ covers every subpath of $p$.
\end{lemma}

\begin{proof}
    For the unit axiom, take $n = 1$, $S_1 = p$ and $s_1 = p$.

    For the composition axiom, suppose $S$ covers $p$, say $p = s_1 \concat \cdots \concat s_n$ with $s_i$ a subpath of $S_i \in S$, and suppose each $s \in S$ is covered by $T_s$.
    Fix $i$ and write $S_i = u_{i,1} \concat \cdots \concat u_{i,m_i}$ with $u_{i,k}$ a subpath of some element of $T_{S_i}$.
    Since $s_i$ is a subpath of $S_i$, \cref{item:window-decomposes} of \cref{lem:window-of-concatenation} presents $s_i$ as a concatenation of subpaths of the $u_{i,k}$, hence of subpaths of elements of $T_{S_i}$, a subpath of a subpath being a subpath.
    Concatenating these presentations over $i = 1, \dots, n$ exhibits $p$ as a concatenation of subpaths of elements of $\bigcup_{s \in S} T_s$, so that union covers $p$.

    For the last claim, let $u$ be a subpath of $p$.
    By \cref{item:window-decomposes} of \cref{lem:window-of-concatenation}, $u$ is a concatenation of subpaths of the $s_i$, hence of subpaths of the $S_i$, so $S$ covers $u$.
\end{proof}

We now record the properties of the entry weight used throughout.

\begin{lemma} \label{lem:entry-weight}
    Let $P$ be a filtered path complex. Then:
    \begin{enumerate}
        \item \label{item:entry-vertices} $l_P(x) = 0$ for every $x \in P_0$;
        \item \label{item:entry-sublevel} $\{\, p \mid l_P(p) < \delta \,\} \subseteq P^\delta \subseteq \{\, p \mid l_P(p) \leq \delta \,\}$ for every $\delta \geq 0$; in particular $l_P(p) < \delta$ implies $p \in P^\delta$;
        \item \label{item:entry-monotone} if $q$ is a subpath of $p$, then $l_P(q) \leq l_P(p)$.
    \end{enumerate}
\end{lemma}

\begin{proof}
    \Cref{item:entry-vertices} holds because every length-$0$ path lies in $P^0$, so $l_P(x) = 0$.

    For \cref{item:entry-sublevel}, the right-hand inclusion is immediate: if $p \in P^\delta$ then $l_P(p) \leq \delta$ by definition of the infimum.
    For the left-hand inclusion, suppose $l_P(p) < \delta$.
    Then, by definition of $l_P$ as an infimum, there is some $\delta'' < \delta$ with $p \in P^{\delta''}$, and since $P^{\delta''} \subseteq P^\delta$ by nesting, $p \in P^\delta$.
    (Note that we do not necessarily have $P^\delta = \{p \mid l_P(p) \leq \delta\}$: equality can fail when the filtration is not a sublevel filtration, for instance if $p$ lies in $P^{\delta'}$ exactly for $\delta' > \delta$. It does hold for the generated filtrations below, but is not needed here.)

    For \cref{item:entry-monotone}, we may assume $l_P(p) < \infty$.
    For every $\delta > l_P(p)$ we have $p \in P^\delta$ by \cref{item:entry-sublevel}, and since $P^\delta$ is a path complex, every subpath $q$ of $p$ lies in $P^\delta$, so $l_P(q) \leq \delta$.
    Letting $\delta \downarrow l_P(p)$ gives $l_P(q) \leq l_P(p)$.
\end{proof}

Conspicuously absent from \cref{lem:entry-weight} is subadditivity of $l_P$ along covers.
That property does hold for the filtrations generated by a weight, which are the ones the later sections use, and we record it in \cref{lem:generated-subadditive} below; it fails for filtered path complexes in general.

\begin{remark} \label{rmk:no-general-subadditivity}
    The obstruction is that a path complex need not be closed under concatenation: \cref{def:path-complex} closes allowed paths under deleting an endpoint, and imposes nothing about gluing them back together.
    For a filtered instance, let
    \[
        P^\delta = \{0,1,2,(0,1),(1,2)\} \quad (\delta < 1), \qquad P^\delta = \{0,1,2,(0,1),(1,2),(0,1,2)\} \quad (\delta \geq 1).
    \]
    Both stages are path complexes, so this is a filtered path complex, and $\{(0,1),(1,2)\}$ covers $(0,1,2)$, since $(0,1,2) = (0,1) \concat (1,2)$.
    Yet $l_P(0,1,2) = 1$ while $l_P(0,1) = l_P(1,2) = 0$.
    What survives in general is only the weaker statement supplied by \cref{item:entry-monotone}: if $p$ is a subpath of some element of $S$, then $l_P(p) \leq \sum_{s \in S} l_P(s)$, the entry weights being nonnegative.
\end{remark}

\subsection{Weighted path complexes as examples} \label{subsec:weighted-examples}

%Filtered path complexes typically arise from weighted path complexes.
The prototypical example of a filtered path complex is given by a weight function defined on the paths of a path complex, which we extend to all paths by covering and then take sublevel sets.
%This construction recurs throughout the paper, so we record it once here and reuse it, in \cref{sec:hypergraph-stability,sec:sequence-hypergraph-stability,sec:network-stability}, for hypergraphs and digraphs.

\begin{definition} \label{def:weighted-path-complex}
    A \emph{weighted path complex} is a pair $(P, w)$ where $P$ is a path complex and $w \colon P \to \mathbb{R}_{\geq 0} \cup \{\infty\}$ satisfies $w(x) = 0$ for every $x \in P_0$.
\end{definition}

Applying \cref{con:generated-filtration} to the path-connected covering relation of \cref{def:path-connected} gives the induced weight
\[
    \overline{w}(p) = \inf\Big\{ \textstyle\sum_{s \in S} w(s) \ \Big\vert \  S \subseteq P \text{ a path connected cover of } p \Big\},
\]
with the convention $\inf \varnothing = \infty$, and the filtration generated by $(P,w)$, with entry weight $l_P = \overline{w}$.
The one point special to this setting is that each $P^\delta$ is a path complex: by the last clause of \cref{lem:path-covering-relation}, any cover of $p$ also covers each subpath of $p$, so $\overline{w}$ is monotone under subpaths and every subpath of $p$ again lies in $P^\delta$.

For these generated filtrations the subadditivity missing from \cref{lem:entry-weight} is available.
It is not needed for \cref{thm:path-stability}, whose proof uses only \cref{item:entry-sublevel}; we record it because it is the correct general form of the property, and because the same argument recurs in Step 3 of \cref{sec:hypergraph-stability,sec:sequence-hypergraph-stability,sec:network-stability} for the weights appearing there.

\begin{lemma} \label{lem:generated-subadditive}
    Let $P$ be the filtered path complex generated by a weighted path complex $(P,w)$.
    Then $l_P(p) \leq \sum_{s \in S} l_P(s)$ for every $S$ covering $p$.
\end{lemma}
\begin{proof}
    By \cref{lem:path-covering-relation}, \cref{def:path-connected} is a covering relation, so \cref{lem:covering-idempotent} applies.
    By \cref{con:generated-filtration} the entry weight is the induced weight, $l_P = \overline{w}$, which is covering-subadditive by \cref{item:cov-subadditive} of \cref{lem:covering-idempotent}.
\end{proof}

The induced weight need not agree with $w$: if $p$ admits a path connected cover $S$ with $\sum_{s \in S} w(s) < w(p)$, then $\overline{w}(p) < w(p)$.
It replaces $w$ by the cheapest cost of assembling a path from covering pieces, which is what makes the sublevel sets path complexes.
When no confusion arises, we denote the generated filtered path complex again by $P$ and its entry weight by $\overline{w} = l_P$.

\begin{remark} \label{rmk:not-injective}
    Distinct weighted path complexes may generate the same filtered path complex.
    Indeed, by idempotency the pairs $(P, w)$ and $(P_0^{<\infty}, \overline{w})$ always do.
    Consequently, no invariant of the generated filtered path complex, including the distance of \cref{def:path-distance} and the persistence modules of \cref{subsec:path-stability}, can distinguish two weighted path complexes with the same induced weight.
    This is the reason the distance below is only a pseudometric.
\end{remark}

\subsection{Correspondences and the path complex distance}

\newcommand{\Corr}{\mathrm{Corr}}

Our goal is to define a version of the interleaving distance between filtered path complexes, following the approach outlined in \cite{chazal-desilva-oudot:stability-geometric-complexes}.
To do so, we briefly recall the standard material on multivalued maps and correspondences.
First, recall that a multivalued map $C \colon A \rightrightarrows B$ between sets is a subset $C \subseteq A \times B$ such that $\pi_1(C) = A$, i.e., the image of $C$ under the projection onto $A$ is $A$.
The transpose of a multivalued map $C \colon A \rightrightarrows B$ is a subset $C^T \subseteq B \times A$ given by $C^T = \{ (b, a) \in B \times A \mid (a, b) \in C \}$.
A multivalued map $C$ is a correspondence if $\pi_2 (C) = B$, or, equivalently, if $C^T$ is again a multivalued map.
The set of all correspondences between $A$ and $B$ is denoted $\Corr(A, B)$.
A map $f \colon A \to B$ is subordinate to $C$ if $(a,f(a))\in C$ for all $a \in A$.
An \emph{ordered lift} through a correspondence $C \colon A \rightrightarrows B$ is a sequence $\sigma = ((a_0,b_0),\dots,(a_n,b_n))$ such that $(a_i,b_i) \in C$ for all $i$; we write $\pi_1(\sigma) = (a_0, \dots, a_n)$ and $\pi_2(\sigma) = (b_0,\dots,b_n)$ for its projections.
In what follows, these notions are applied with $A$ and $B$ the vertex sets of two path complexes.

\begin{definition} \label{def:path-distance}
Let $P$ and $Q$ be filtered path complexes.
\begin{itemize}
    \item The \emph{path distortion} $\dispath(C)$ of a correspondence $C \colon P_0 \rightrightarrows Q_0$ is given by
\[
\dispath(C) = \sup \Big\{ \big|l_P(\pi_1(\sigma)) - l_Q(\pi_2(\sigma)) \big| \ \Big| \ \sigma \text{ is an ordered lift through } C \Big\} \text{.}
\]
    \item The \emph{path complex distance} between $P$ and $Q$ is given by
\[
\dpathcomplex{P}{Q} = \frac{1}{2} \inf \Big\{ \dispath(C) \ \Big| \ C \in \Corr(P_0,Q_0) \Big\} \text{.}
\]
\end{itemize}
\end{definition}

Since $l_P$ is defined on all elementary paths, the expression above makes sense for every ordered lift, and it takes values in $\mathbb{R}_{\geq 0} \cup \{\infty\}$.
For weighted path complexes $(P,w_P)$ and $(Q,w_Q)$, we let $\dpathcomplex{P}{Q}$ denote the distance between the associated filtered path complexes.%; equivalently, it is computed from the completed weights $\overline{w}_{P}$ and $\overline{w}_{Q}$.

Our goal is to prove stability of persistent path homology, that is, that the interleaving distance between path homologies of $P$ and $Q$ is bounded above by their path complex distance.
Before doing this, we prove that the path complex distance defines an extended pseudometric.

\begin{theorem} \label{thm:path-distance-pseudometric}
   The path complex distance $\dpathcomplex{\cdot}{\cdot}$ is an extended pseudometric on filtered path complexes.
\end{theorem}

\begin{proof}
    Nonnegativity is immediate, since $\dispath(C) \ge 0$, and symmetry follows from $\dispath(C^T) = \dispath(C)$.
    If $P = Q$, then the identity correspondence has distortion $0$, so $\dpathcomplex{P}{Q} = 0$.
    It remains to prove the triangle inequality.

    Let $P$, $Q$, $R$ be filtered path complexes.
    We want to show that
    $\dpathcomplex{P}{R} \leq \dpathcomplex{P}{Q} + \dpathcomplex{Q}{R}$.
    Let $C \colon P_0 \rightrightarrows Q_0$ and $D \colon Q_0 \rightrightarrows R_0$ be correspondences, with composite $D \circ C \colon P_0 \rightrightarrows R_0$.
    It suffices to show
    \[
    \dispath(D \circ C) \leq \dispath(C) + \dispath(D).
    \]
    Let $\sigma = ((x_0,z_0),\dots,(x_n,z_n))$ be an ordered lift through $D \circ C$.
    By definition of $D \circ C$, for all $i$ there exists $y_i \in Q_0$ such that $(x_i,y_i) \in C$ and $(y_i,z_i) \in D$.
    So, we get ordered lifts $\sigma_C = ((x_0,y_0), \dots,(x_n,y_n))$ and $\sigma_D = ((y_0,z_0),\dots,(y_n,z_n))$.
    Let $p = (x_0,\dots,x_n)$, $q=(y_0,\dots,y_n)$, and $r =(z_0,\dots,z_n)$. Then by the usual triangle inequality over $\mathbb{R}$,
    \[
    | l_P(p) - l_R(r)| \leq |l_P(p) - l_Q(q)| + |l_Q(q) - l_R(r)|.
    \]
    Since $\sigma_C$ and $\sigma_D$ are ordered lifts,
    $|l_P(p) - l_Q(q)| \leq \dispath(C)$ and $|l_Q(q) - l_R(r)| \leq \dispath(D)$.
    Hence
    \[
     | l_P(p) - l_R(r)| \leq \dispath(C) + \dispath(D).
    \]
    Since $\sigma$ was an arbitrary ordered lift, we can take the supremum over all such $\sigma$ to get
    \[
    \dispath(D \circ C) \leq \dispath(C) + \dispath(D).
    \]
    Finally, we can take the infimum over all correspondences $C$ and $D$ to get
    \[
    \dpathcomplex{P}{R} \leq \dpathcomplex{P}{Q} + \dpathcomplex{Q}{R}.
    \]
    So the triangle inequality holds, and $\dpathcomplex{\cdot}{\cdot}$ is an extended pseudometric.
\end{proof}

Note that $\dpathcomplex{P}{Q} = 0$ does not imply $P = Q$; by \cref{rmk:not-injective}, weighted path complexes with the same induced weight generate the same filtered path complex.

\subsection{Stability} \label{subsec:path-stability}

Now that we have an extended pseudometric on filtered path complexes, we would like to prove stability.
%To do this, we follow the framework of \cite{chazal-desilva-oudot:stability-geometric-complexes} by proving the following propositions.
This is established in a series of lemmas and propositions, exactly as in \cite{chazal-desilva-oudot:stability-geometric-complexes}.

\begin{definition} \label{def:path-mvm}
    Let $P = \{ P^\delta\}_{\delta \geq0}$ and $Q = \{ Q^\delta\}_{\delta \geq0}$ be filtered path complexes.
    \begin{itemize}
        \item A multivalued map $C \colon P_0 \rightrightarrows Q_0$ is an \emph{$\varepsilon$-path-multivalued map} from $P$ to $Q$ if for every $\delta \geq 0$, every path $(x_0,\dots,x_k) \in P^\delta$, and every choice of vertices $y_i \in C(x_i)$, the sequence $(y_0,\dots,y_k)$ belongs to $Q^{\delta + \varepsilon}$.
        \item A correspondence $C \colon P_0 \rightrightarrows Q_0$ is an \emph{$\varepsilon$-path-correspondence} if both $C$ and $C^T$ are $\varepsilon$-path-multivalued maps.
    \end{itemize}
\end{definition}

\begin{lemma} \label{lem:subordinate-length-one}
Let $C \colon P_0 \rightrightarrows Q_0$ be an $\varepsilon$-path-multivalued map from $P$ to $Q$.
If $f_1,f_2 \colon P_0 \to Q_0$ are two subordinate maps of $C$, then for every $x \in P_0$ with $(x) \in P^\delta$ we have $(f_1(x), f_2(x)) \in Q^{\delta + \varepsilon}$.
\end{lemma}
\begin{proof}
    Let $p = (x,x) \in P^\delta$, then for any $y_1,y_2 \in C(x)$, by definition we know that $(y_1,y_2) \in Q^{\delta + \varepsilon}$, as required.
\end{proof}

\begin{proposition} \label{prop:path-mvm-homology}
Let $C \colon P_0 \rightrightarrows Q_0$ be an $\varepsilon$-path-multivalued map from $P$ to $Q$.
Then $C$ induces a canonical linear map on homology $\hlgy{C} \in \Hom^\varepsilon(\hlgy{P},\hlgy{Q})$, equal to $\hlgy{f}$ for any $f$ subordinate to $C$.
\end{proposition}
 
\begin{proof}
    Let $f \colon P_0 \to Q_0$ be any subordinate map to $C$.
    Then for every $\delta \geq 0$, there is an induced path complex map $f^\delta \colon P^\delta \to Q^{\delta + \varepsilon}$ given by, for $p=(x_0, \dots , x_n )\in P^\delta$,
    \[
    f^\delta(p) = (f(x_0), \dots, f(x_n) ),
    \]
    that commutes with the inclusions $\iota_\delta^{\delta'} \colon P^\delta \to P^{\delta'}$ and $\tau_{\delta + \varepsilon}^{\delta' + \varepsilon} \colon Q^{\delta + \varepsilon} \to Q^{\delta' + \varepsilon}$, i.e., the square 
\[
    \begin{tikzcd}[column sep = large]
    P^\delta \arrow[r,"\iota_\delta^{\delta'}"] \arrow[d,"f^\delta"'] & P^{\delta '} \arrow[d,"f^{\delta '}"] \\
    Q^{\delta + \varepsilon} \arrow[r,"\tau_{\delta + \varepsilon}^{\delta ' + \varepsilon}"] & Q^{\delta ' + \varepsilon}
    \end{tikzcd}
\]
commutes for all $\delta \leq \delta'$.
By functoriality of path homology, we have $\hlgy{\tau_{\delta + \varepsilon}^{\delta ' + \varepsilon}} \circ \hlgy{f^\delta} = \hlgy{f^{\delta'}} \circ \hlgy{\iota_\delta^{\delta'}}$, which proves that $\hlgy{f^\delta} \in \Hom^\varepsilon(\hlgy{P},\hlgy{Q})$.
 
Now consider any two $f_1,f_2 \colon P_0 \to Q_0$ subordinate to $C$, with induced path complex maps $f^\delta_1, f_2^\delta \colon P^\delta \to Q^{\delta + \varepsilon}$.
We show that $f_1^\delta \simeq f_2^\delta$ by exhibiting a one-step homotopy.
Recalling that $\Cyl(P^\delta) = P^\delta \square I$ has vertex set $P^\delta_0 \times \{0,1\}$, we define $\alpha \colon \Cyl(P^{\delta}) \to Q^{\delta + \varepsilon}$ by
\[
    (x,0) \mapsto f_1(x), \qquad (x,1) \mapsto f_2(x).
\]
By construction $\alpha \circ i_0 = f_1^\delta$ and $\alpha \circ i_1 = f_2^\delta$, so once $\alpha$ is shown to be a well-defined map path complex, it is a one-step homotopy from $f_1^\delta$ to $f_2^\delta$.

It remains to check that $\alpha$ carries each path of $\Cyl(P^\delta)$ to a path of $Q^{\delta+\varepsilon}$.
Such a path has collapsed $P$-projection some $(x_0,\dots,x_m) \in P^\delta$, and its $I$-coordinate is nondecreasing, so it is either the image of $i_0$ or of $i_1$, or else has the form
\[
    ((x_0,0),\dots,(x_k,0),(x_k,1),\dots,(x_m,1))
\]
for some $0 \leq k \leq m$.
Its image under $\alpha$ is accordingly $(f_1(x_0),\dots,f_1(x_m))$, or $(f_2(x_0),\dots,f_2(x_m))$, or
\[
    (f_1(x_0),\dots,f_1(x_k),f_2(x_k),\dots,f_2(x_m)),
\]
and since $f_1$ and $f_2$ are subordinate to $C$, every entry of the image lies in $C$ of the corresponding vertex of $(x_0,\dots,x_m)$.
The first two images are $f_1^\delta$ and $f_2^\delta$ applied to $(x_0,\dots,x_m)$, so they lie in $Q^{\delta+\varepsilon}$.
In the remaining case the transition edge $(f_1(x_k),f_2(x_k))$ lies in $Q^{\delta+\varepsilon}$ by \cref{lem:subordinate-length-one}, and since $(x_0,\dots,x_m) \in P^\delta$, the defining property of an $\varepsilon$-path-multivalued map places the whole image in $Q^{\delta+\varepsilon}$.
Thus $\alpha$ is a well-defined map of path complexes, hence a one-step homotopy $f_1^\delta \simeq f_2^\delta$.
Therefore $\hlgy{f_1} = \hlgy{f_2}$ by \cref{thm:homotopic-equal-on-hlgy}, so the map $\hlgy{C}$ is well-defined.
\end{proof}

As a consequence of \cref{prop:path-mvm-homology}, all subsets of $C$ that are themselves multivalued maps induce the same map on homology.

\begin{proposition} \label{prop:path-correspondence-subset}
    If $C' \subseteq C \colon P_0 \rightrightarrows Q_0$ and $C$ is an $\varepsilon$-path-multivalued map from $P$ to $Q$, then $C'$ is an $\varepsilon$-path-multivalued map from $P$ to $Q$ and $\hlgy{C'} = \hlgy{C}$.
\end{proposition}
\begin{proof}
    Since $C'$ is a multivalued map contained in $C$, it is also an $\varepsilon$-path-multivalued map, and any map $f \colon P_0 \to Q_0$ that is subordinate to $C'$ is also subordinate to $C$, so from \cref{prop:path-mvm-homology} we have $\hlgy{C'} = \hlgy{f} = \hlgy{C}$.
\end{proof}

\begin{proposition} \label{prop:path-composition-mvm}
    Let $P$, $Q$, $R$ be filtered path complexes.
    If
    \begin{center}
        $C \colon P_0 \rightrightarrows Q_0$ is an $\varepsilon$-path-multivalued map from $P$ to $Q$, \\
        $D \colon Q_0 \rightrightarrows R_0$ is a $\eta$-path-multivalued map from $Q$ to $R$,
    \end{center}
    then the composite $D \circ C \colon P_0 \rightrightarrows R_0$ is a $(\varepsilon + \eta)$-path-multivalued map from $P$ to $R$, and $\hlgy{D \circ C} = \hlgy{D} \circ \hlgy{C}$.
\end{proposition}

\begin{proof}
We first show that $D \circ C$ is an $(\varepsilon + \eta)$-path-multivalued map.
Let $p=(x_0,\dots,x_n) \in P^\delta$. Since $C$ is an $\varepsilon$-path-multivalued map, for every choice $y_i \in C(x_i)$ the path $p'=(y_0,\dots, y_n) \in Q^{\delta + \varepsilon}$.
Now apply $D$.
Since $D$ is an $\eta$-path-multivalued map, for every choice $z_i \in D(y_i)$ the path $p''=(z_0,\dots, z_n) \in R^{\delta + \varepsilon + \eta}$.
Since $y_i \in C(x_i)$ and $z_i \in D(y_i)$, we get that $z_i \in (D \circ C)(x_i)$.
Therefore, $D \circ C$ is an $(\varepsilon + \eta)$-path-multivalued map.

Let $f \colon P_0 \to Q_0$ be subordinate to $C$, and let $g \colon Q_0 \to R_0$ be subordinate to $D$.
The composite $g \circ f \colon P_0 \to R_0$ is subordinate to $D \circ C$, thus we have
\[
\hlgy{D \circ C} = \hlgy{g \circ f}= \hlgy{g} \circ \hlgy{f} = \hlgy{D} \circ \hlgy{C}
\]
as required.
\end{proof}

Now, from \cref{prop:path-correspondence-subset} and \cref{prop:path-composition-mvm} we deduce the following.

\begin{proposition} \label{prop:path-interleaving}
    Let $P$ and $Q$ be filtered path complexes.
    If $C \colon P_0 \rightrightarrows Q_0$ is a correspondence such that $C$ and $C^T$ are both $\varepsilon$-path-multivalued maps, then together they induce a canonical $\varepsilon$-interleaving between $\hlgy{P}$ and $\hlgy{Q}$, the interleaving homomorphisms being $\hlgy{C}$ and $\hlgy{C^T}$.
\end{proposition}
\begin{proof}
    Since $\identity{P_0} \subseteq C^T \circ C$ and $\identity{Q_0} \subseteq C \circ C^T$, from \cref{prop:path-composition-mvm} we know that $\identity{P_0}$ and $\identity{Q_0}$ are both $2\varepsilon$-path-multivalued maps.
    So it follows that $\identity{\hlgy{P}}^{2 \varepsilon} = \hlgy{\identity{P_0}} = \hlgy{C^T} \circ \hlgy{C}$ and $\identity{\hlgy{Q}}^{2 \varepsilon}= \hlgy{\identity{Q_0}} = \hlgy{C} \circ \hlgy{C^T}$.
    Therefore, $\hlgy{P}$ and $\hlgy{Q}$ are $\varepsilon$-interleaved.
\end{proof}

Finally, using \cref{prop:path-interleaving}, we prove stability of persistent path homology of path complexes.

\begin{theorem} \label{thm:path-stability}
Let $P$ and $Q$ be filtered path complexes, and let $k \in \mathbb{N}$.
Then
\[
    \dbottleneck{\dgm_k(P)}{\dgm_k(Q)} \leq 2\dpathcomplex{P}{Q}.
\]
In particular, passing to generated filtered path complexes, the same bound holds for any two weighted path complexes, applied to the filtered path complexes they generate.
\end{theorem}
 
\begin{proof}
Let $\eta > 2\dpathcomplex{P}{Q}$.
    Then there exists a correspondence $C \colon P_0 \rightrightarrows Q_0$ such that $\dispath(C) < \eta$.
 
    We first show that $C$ is an $\eta$-path-multivalued map from $P$ to $Q$.
    Let $\delta\geq 0$, and let $p=(x_0,\dots,x_n)\in P^\delta$, so that $l_P(p) \leq \delta$.
    Choose arbitrary vertices $y_i\in C(x_i)$ and set $p'=(y_0,\dots,y_n)$.
    Then $\sigma=((x_0,y_0),\dots,(x_n,y_n))$ is an ordered lift through $C$, with
    $\pi_1(\sigma)=p$ and $\pi_2(\sigma)=p'$.
    Since $\dispath(C)<\eta$, we have
    \[
    \left| l_P(p)- l_Q(p') \right| <\eta.
    \]
    Since $l_P(p) \leq \delta$, we have
    \[
    l_Q(p') < l_P(p) + \eta \leq \delta + \eta .
    \]
    Thus $p'\in Q^{\delta + \eta}$ by \cref{item:entry-sublevel} of \cref{lem:entry-weight}, and $C$ is an $\eta$-path-multivalued map from $P$ to $Q$.
    The same argument applied to $C^T$ shows that $C$ is an $\eta$-path-correspondence.
 
    By \cref{prop:path-interleaving}, since $C$ is an $\eta$-path-correspondence, the persistence modules $\hlgy{P}$ and $\hlgy{Q}$ are $\eta$-interleaved.
    By the algebraic stability theorem, $\dbottleneck{\dgm_k(P)}{\dgm_k(Q)} \leq \eta$.
    Since $\eta$ was any value exceeding $2\dpathcomplex{P}{Q}$, we conclude
    \[
    \dbottleneck{\dgm_k(P)}{\dgm_k(Q)} \leq 2\dpathcomplex{P}{Q},
    \]
    as required.
\end{proof}

The later sections all use \cref{thm:path-stability} in the same way: an object is replaced by an associated path complex, and a comparison of distances is fed into the theorem.

\begin{corollary} \label{cor:transfer}
    Let $G$ and $H$ be weighted objects whose persistent path homology is by definition that of associated path complexes $P(\cmpl{G})$ and $P(\cmpl{H})$, so that $\dgm_k(G) = \dgm_k(P(\cmpl{G}))$ and $\dgm_k(H) = \dgm_k(P(\cmpl{H}))$ for all $k \in \mathbb{N}$.
    Let $d$ be a distance on such objects satisfying
    \[
        \dpathcomplex{P(\cmpl{G})}{P(\cmpl{H})} \leq d(G,H).
    \]
    Then for every $k \in \mathbb{N}$,
    \[
        \dbottleneck{\dgm_k(G)}{\dgm_k(H)} \leq 2\, d(G,H).
    \]
\end{corollary}
\begin{proof}
    Unwinding the definition of $\dgm_k$ on objects and applying \cref{thm:path-stability} gives
    \begin{align*}
        \dbottleneck{\dgm_k(G)}{\dgm_k(H)}
        &= \dbottleneck{\dgm_k(P(\cmpl{G}))}{\dgm_k(P(\cmpl{H}))} && \text{(definition of $\dgm_k$ on objects)} \\
        &\leq 2\dpathcomplex{P(\cmpl{G})}{P(\cmpl{H})} && \text{(\cref{thm:path-stability})} \\
        &\leq 2\, d(G,H). && \text{(hypothesis on $d$)} \qedhere
    \end{align*}
\end{proof}
\section{Stability of Persistent Path Homology of Hypergraphs} \label{sec:hypergraph-stability}
In this section we deduce stability of persistent path homology of hypergraphs, following the four steps outlined in \cref{subsec:common-pattern}.
Our review of path homology of hypergraphs is brief, and we refer the reader to \cite{grigor2019homology} for a thorough treatment.
As in \cref{sec:path-cpx-stability}, filtered objects are primitive and weighted ones are a source of examples.

\begin{definition} \label{def:filtered-hypergraph}
A \emph{hypergraph} is a pair $G = (X, E_G)$, where $X$ is a vertex set and $E_G \subseteq \subsets{X}$ is a set of edges, where $\subsets{X}$ denotes the set of nonempty subsets of $X$.
A \emph{filtered hypergraph} is a family $G = \{G^\delta\}_{\delta \geq 0}$ of hypergraphs on a fixed vertex set $X$, with inclusions $G^\delta \hookrightarrow G^{\delta'}$ for $\delta \leq \delta'$.
Its \emph{entry weight} $l_G \colon \subsets{X} \to \mathbb{R}_{\geq 0} \cup \{\infty\}$ is $l_G(e) = \inf\{\delta \geq 0 \mid e \in E_G^\delta\}$.
\end{definition}

\begin{definition} \label{def:weighted-hypergraph}
A \emph{weighted hypergraph} is a triple $G = (X, E_G, w_G)$, where $(X, E_G)$ is a hypergraph and $w_G \colon E_G \to \mathbb{R}_{\ge 0} \cup \{ \infty \}$ is a weight function.
\end{definition}

\subsection{Step 1: The covering relation} \label{subsec:hyp-step1}

The covering relation for hypergraphs is generated by connectedness of edges.

\begin{definition} \label{def:hypergraph-connected}
    Let $G = (X,E_G)$ be a hypergraph. A subset $S \subseteq E_G$ is \emph{connected} if for all $u,v \in \bigcup S$ there exist vertices $u=v_0, v_1, \dots, v_m=v$ and edges $e_1, \dots, e_m \in S$ with $\{v_{i-1}, v_i\} \subseteq e_i$ for each $1 \leq i \leq m$.
    We say that $S$ \emph{covers} a subset $s \subseteq X$ if $S$ is connected and $s \subseteq \bigcup S$.
\end{definition}

This is a covering relation on $\subsets{X}$ in the sense of \cref{def:covering-relation}: the singleton $\{s\}$ covers $s$, and if $S$ covers $s$ while each $e \in S$ is covered by $T_e$, then $\bigcup_e T_e$ is again connected, since consecutive members of $S$ share a vertex.
Applying \cref{con:generated-filtration} to this covering relation gives the induced weight
\[
    \overline{w}_{G}(s) = \inf\Big\{ \textstyle\sum_{e \in S} w_G(e) \;\Big|\; S \text{ covers } s \Big\}
\]
and the filtered hypergraph \emph{generated by} $G$, with entry weight $l_G = \overline{w}_{G}$; each sublevel set is a set of subsets of $X$, hence again a hypergraph.

\begin{definition} \label{def:hypergraph-distance}
    Let $G= (X, E_G, w_G)$ and $H = (Y,E_H,w_H)$ be weighted hypergraphs.
    The \emph{hypergraph distortion} of a correspondence $C \colon X \rightrightarrows Y$ is
    \[
        \dishyper(C) = \sup_\sigma \big| \overline{w}_{G}(\pi_1(\sigma)) - \overline{w}_{H}(\pi_2(\sigma)) \big|,
    \]
    where $\sigma$ ranges over finite ordered lifts through $C$, and the \emph{hypergraph distance} between $G$ and $H$ is
    \[
        \dhyper{G}{H} = \tfrac{1}{2}\inf \{ \dishyper(C) \mid C \in \Corr(X,Y)\}.
    \]
\end{definition}

\subsection{Step 2: The associated path complex} \label{subsec:hyp-step2}

We recall how a hypergraph gives rise to a path complex, and hence to path homology.

\begin{definition} \label{def:density-q-path}
    Let $G = (X,E_G,w_G)$ be a weighted hypergraph and $q \geq 1$.
    Following Grigor'yan, Jimenez, Muranov, and Yau \cite{grigor2019homology}, the \emph{density-$q$ path complex} $P^q(G)$ is the weighted path complex on $X$ whose weight function is
    \[
    w_{P^q(G)}(p) = \sup\big\{ \inf\{w_G(e) \mid b\subseteq e \in E_G \} \;\big|\; b = (x_i, \dots, x_{i+q-1}) \text{ a } q\text{-block of } p \big\}.
    \]
    Here a \emph{$q$-block} of $p$ is a subpath of $p$ on $q$ vertices; a path on fewer than $q$ vertices has none, and the supremum over the empty set is $0$.
\end{definition}

When $E_G = \subsets{X}$, or when the induced weight $\overline{w}_{G}$ is used in place of $w_G$, each infimum is attained at $e = b$ and the weight simplifies to
    \[
    w_{P^q(G)}(p) = \sup\{ \overline{w}_{G}(b) \mid b \text{ a } q\text{-block of } p \}.
    \]

\begin{example} \label{ex:density-q}
The weighted hypergraph with edges $(\{1,2\},1), (\{2,3\},2), (\{4,5\},1)$, illustrated below, yields the density-$2$ weighted path complex
    \[
    ((1,2),1),\ ((2,3),2),\ ((1,2,3),2),\ ((4,5),1).
    \]
\begin{center}
\begin{tikzpicture}[
  vtx/.style={circle, fill=black, inner sep=1.5pt},
  bub/.style={line cap=round, line join=round, opacity=0.30},
]
  \draw[bub, blue,   line width=6.5mm] (0,0.9) -- (1.4,0);
  \draw[bub, orange, line width=6.5mm] (1.4,0) -- (2.8,0.9);
  \draw[bub, teal,   line width=6.5mm] (4.6,0.45) -- (6.0,0.45);
  \node[vtx] at (0,0.9) {};    \node at (-0.15,1.25) {$1$};
  \node[vtx] at (1.4,0) {};    \node at (1.4,-0.55) {$2$};
  \node[vtx] at (2.8,0.9) {};  \node at (2.95,1.25) {$3$};
  \node[vtx] at (4.6,0.45) {}; \node at (4.6,0.95) {$4$};
  \node[vtx] at (6.0,0.45) {}; \node at (6.0,0.95) {$5$};
  \node[blue]   at (0.05,0.05) {$1$};
  \node[orange] at (2.75,0.05) {$2$};
  \node[teal]   at (5.3,-0.35) {$1$};
\end{tikzpicture}
\end{center}
The path $(1,2,3)$ has weight $2$: its $2$-blocks $(1,2)$ and $(2,3)$ have weights $1$ and $2$, and the supremum is $2$.
\end{example}

\subsection{Step 3: The induced path weight} \label{subsec:hyp-step3}

A weighted hypergraph yields a weighted path complex in two ways: convert to a density-$q$ path complex and then induce a weight by covering, which gives $\cmpl{P^q(G)}$, or induce the hypergraph weight first and then convert, which gives $P^q(\cmpl{G})$.
These do not agree in general:
\[
    \cmpl{P^q(G)} \neq P^q(\cmpl{G}).
\]
The two constructions do agree when applied to $\cmpl{G}$ for $q \leq 2$, and in that range one inequality relates the composites.
Neither statement is used in Step 4 or in the stability theorem, both of which hold for every $q$; we record them because they explain how the two constructions are related.

\begin{proposition} \label{prop:hyp-idempotent}
    Let $G = (X,E_G,w_G)$ be a weighted hypergraph.
    Then $\overline{w}_{P^q(\cmpl{G})} \leq w_{P^q(\cmpl{G})}$, with equality when $q \leq 2$.
\end{proposition}
\begin{proof}
    The inequality is \cref{item:cov-decreases} of \cref{lem:covering-idempotent}.
    Suppose $q \leq 2$, write $w = w_{P^q(\cmpl{G})}$, and let $S$ cover a path $p$, say $p = s_1 \concat \cdots \concat s_n$ with $s_i$ a subpath of $S_i \in S$.
    A $q$-block $b$ of $p$ is a subpath of $p$ on $q \leq 2$ vertices, so by \cref{item:window-short} of \cref{lem:window-of-concatenation} it is a subpath of a single $s_i$, hence of $S_i$; and a subpath on $q$ vertices is a $q$-block.
    Since $w(S_i)$ is the supremum over the $q$-blocks of $S_i$, one of which is $b$,
    \[
        \inf\{\overline{w}_{G}(e) \mid b \subseteq e \in E_G\} \leq w(S_i) \leq \sum_{s \in S} w(s),
    \]
    the second inequality because $w$ is nonnegative.
    Taking the supremum over the $q$-blocks of $p$ gives $w(p) \leq \sum_{s \in S} w(s)$, so $w$ is covering-subadditive.
    By \cref{lem:covering-idempotent}, $\overline{w}$ is the largest covering-subadditive weight below $w$, whence $w \leq \overline{w}$.
\end{proof}

\begin{remark} \label{rmk:q-at-least-three}
    Equality fails for $q \geq 3$.
    Take $X = \{0,1,2,3\}$ with every nonempty subset an edge, of weight $5$ if it contains $3$ and $0$ otherwise.
    Any connected family whose union contains a vertex set meeting $3$ must include an edge containing $3$, so this weight is covering-subadditive and $G = \cmpl{G}$.
    With $q = 3$ we have $w_{P^3(G)}((0,1,2,3)) = 5$, contributed by the block $(1,2,3)$, while $(0,1,2)$ has weight $0$ and $(2,3)$ has weight $0$ for want of any $3$-block at all.
    Since $(0,1,2,3) = (0,1,2) \concat (2,3)$, the family $\{(0,1,2),(2,3)\}$ covers $(0,1,2,3)$, so $\overline{w}_{P^3(G)}((0,1,2,3)) = 0 < 5$.
    The obstruction is that a $q$-block of $p$ may straddle the junction between two pieces of a cover and so be a block of neither.
    For $q \leq 2$ a block has at most two vertices, and consecutive pieces already share a vertex, so no block can straddle; this is exactly \cref{item:window-short} of \cref{lem:window-of-concatenation}.
\end{remark}

As advertised, the two composites are related by an inequality.

\begin{proposition} \label{prop:hyp-composite-inequality}
Let $G=(X,E_G,w_G)$ be a weighted hypergraph and let $q \leq 2$. Then $w_{P^q(\cmpl{G})} \leq \overline{w}_{P^q(G)}$.
\end{proposition}
\begin{proof}
By \cref{item:cov-decreases} of \cref{lem:covering-idempotent}, $\overline{w}_{G} \leq w_G$, so for each $q$-block $b$,
\[
\inf \{ \overline{w}_{G}(e) \mid b \subseteq e \in E_G\} \leq \inf \{ w_G(e) \mid b \subseteq e \in E_G\}.
\]
Taking the supremum over the $q$-blocks of $p$ gives $w_{P^q(\cmpl{G})}(p) \leq w_{P^q(G)}(p)$; that is, $w_{P^q(\cmpl{G})} \leq w_{P^q(G)}$ pointwise.
Completing both sides and using that completion is order-preserving and that $w_{P^q(\cmpl{G})}$ is already complete (\cref{prop:hyp-idempotent}),
\[
w_{P^q(\cmpl{G})} = \overline{w}_{P^q(\cmpl{G})} \leq \overline{w}_{P^q(G)}. \qedhere
\]
\end{proof}

The inequality can be strict as the following example shows.

\begin{example} \label{ex:strict-inequality}
    Take $q=2$ and the weighted hypergraph $G$ on $X = \{1,2,3,4\}$ with edges
    \[
    (\{1,2,3\}, 1),\quad (\{2,3,4 \}, 5),\quad (\{1,4\}, 1),
    \]
    illustrated below, and consider the path $(3,4)$.
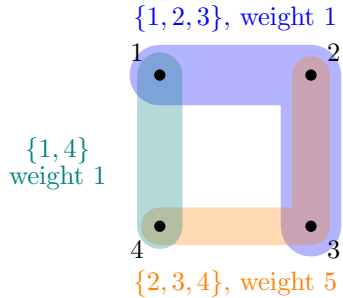
\begin{center}
\begin{tikzpicture}[
  vtx/.style={circle, fill=black, inner sep=1.5pt},
  bub/.style={line cap=round, line join=round, opacity=0.30},
]
  \draw[bub, blue,   line width=8mm]   (0,2) -- (2,2) -- (2,0);
  \draw[bub, orange, line width=5mm]   (2,2) -- (2,0) -- (0,0);
  \draw[bub, teal,   line width=6mm]   (0,2) -- (0,0);
  \node[vtx] at (0,2) {}; \node at (-0.3,2.3) {$1$};
  \node[vtx] at (2,2) {}; \node at (2.3,2.3) {$2$};
  \node[vtx] at (2,0) {}; \node at (2.3,-0.3) {$3$};
  \node[vtx] at (0,0) {}; \node at (-0.3,-0.3) {$4$};
  \node[blue]   at (1,2.75) {$\{1,2,3\}$, weight $1$};
  \node[orange] at (1,-0.75) {$\{2,3,4\}$, weight $5$};
  \node[teal]   at (-1.35,1) {$\{1,4\}$};
  \node[teal]   at (-1.35,0.65) {weight $1$};
\end{tikzpicture}
\end{center}
    Inducing the hypergraph weight first, $\overline{w}_{G}(\{3,4\}) = 2$: the family $\{\{1,2,3\},\{1,4\}\}$ is connected, its union contains $\{3,4\}$, and its cost is $1 + 1 = 2$, which beats the single edge $\{2,3,4\}$ of cost $5$.
    Hence $w_{P^2(\cmpl{G})}((3,4)) = 2$.
    Converting first, $w_{P^2(G)}((3,4)) = 5$, since $\{2,3,4\}$ is the only edge containing $\{3,4\}$.
    No cover lowers this: by \cref{item:window-short} of \cref{lem:window-of-concatenation}, a cover of $(3,4)$ must contain a path having $3$ and $4$ adjacent, and any such path has $(3,4)$ among its $2$-blocks, hence weight at least $5$.
    So $\overline{w}_{P^2(G)}((3,4)) = 5 \neq 2 = w_{P^2(\cmpl{G})}((3,4))$, and therefore $P^2(\cmpl{G}) \neq \cmpl{P^2(G)}$.
\end{example}

\subsection{Step 4: Distance comparison} \label{subsec:hyp-step4}

\begin{proposition} \label{prop:hyp-distance-comparison}
    Let $G = (X,E_G,w_G)$ and $H = (Y,E_H,w_H)$ be weighted hypergraphs. Then
    \[
    \dpathcomplex{P^q(\cmpl{G})}{P^q(\cmpl{H})} \leq \dhyper{G}{H}.
    \]
\end{proposition}
\begin{proof}
Let $C \colon X \rightrightarrows Y$ be a correspondence; we show $\dispath(C) \leq \dishyper(C)$.
Let $\sigma=((x_0,y_0),\dots,(x_n,y_n))$ be an ordered lift through $C$, and set $p=\pi_1(\sigma)$ and $p'=\pi_2(\sigma)$.
Since the weight of $\cmpl{G}$ is already induced, \cref{def:density-q-path} gives
\[
    w_{P^q(\cmpl{G})}(p) = \sup_i \overline{w}_{G}(b_i), \qquad w_{P^q(\cmpl{H})}(p') = \sup_i \overline{w}_{H}(b'_i),
\]
where $b_i = (x_i,\dots,x_{i+q-1})$ and $b'_i = (y_i,\dots,y_{i+q-1})$ are the corresponding $q$-blocks.
For each $i$, the restriction $\sigma_i = ((x_i,y_i),\dots,(x_{i+q-1},y_{i+q-1}))$ is an ordered lift through $C$ with $\pi_1(\sigma_i) = b_i$ and $\pi_2(\sigma_i) = b'_i$, so
\[
    |\overline{w}_{G}(b_i) - \overline{w}_{H}(b'_i)| \leq \dishyper(C).
\]
As $p$ has finitely many $q$-blocks, \cref{lem:sup-comparison} applies and gives
\[
    |w_{P^q(\cmpl{G})}(p) - w_{P^q(\cmpl{H})}(p')| = \Big| \sup_i \overline{w}_{G}(b_i) - \sup_i \overline{w}_{H}(b'_i) \Big| \leq \dishyper(C).
\]
Taking the supremum over ordered lifts gives $\dispath(C) \leq \dishyper(C)$, and the infimum over correspondences gives the claim.
\end{proof}

Feeding this into stability for path complexes yields stability for hypergraphs.

\begin{theorem} \label{thm:hyp-stability}
Let $G = (X,E_G,w_G)$ and $H = (Y,E_H, w_H)$ be weighted hypergraphs, and let $k \in \mathbb{N}$. Then
\[
    \dbottleneck{\dgm_k(G)}{\dgm_k(H)} \leq 2\dhyper{G}{H}.
\]
\end{theorem}
\begin{proof}
By definition, the persistent path homology of a hypergraph is that of its density-$q$ path complex.
The claim is therefore \cref{cor:transfer}, whose hypothesis is \cref{prop:hyp-distance-comparison}.
\end{proof}
\section{Stability of Persistent Path Homology of Sequence Hypergraphs} \label{sec:sequence-hypergraph-stability}
In this section we deduce stability of persistent path homology of sequence hypergraphs, following the four steps outlined in \cref{subsec:common-pattern}.
A sequence hypergraph is a hypergraph whose edges are sequences of distinct vertices, rather than unordered subsets.
As in \cref{sec:hypergraph-stability}, filtered objects are primitive and weighted ones are a source of examples.

\begin{definition} \label{def:sequence-hypergraph}
A \emph{sequence hypergraph} is a pair $G = (X, E_G)$, where $X$ is a vertex set and $E_G \subseteq \mathbb{S}(X)$ is a set of edges, where $\mathbb{S}(X)$ denotes the set of nonempty totally ordered subsets of $X$, that is, nonempty subsets equipped with a linear order on their elements.
A \emph{filtered sequence hypergraph} is a family $G = \{G^\delta\}_{\delta \geq 0}$ of sequence hypergraphs on a fixed vertex set $X$, with inclusions $G^\delta \hookrightarrow G^{\delta'}$ for $\delta \leq \delta'$.
Its \emph{entry weight} $l_G \colon \mathbb{S}(X) \to \mathbb{R}_{\geq 0} \cup \{\infty\}$ is $l_G(e) = \inf\{\delta \geq 0 \mid e \in E_G^\delta\}$.
\end{definition}

\begin{definition} \label{def:weighted-sequence-hypergraph}
A \emph{weighted sequence hypergraph} is a triple $G = (X, E_G, w_G)$, where $(X, E_G)$ is a sequence hypergraph and $w_G \colon E_G \to \mathbb{R}_{\ge 0} \cup \{ \infty \}$ is a weight function.
\end{definition}

\subsection{Step 1: The covering relation} \label{subsec:dhyp-step1}

The covering relation for sequence hypergraphs is generated by \emph{sequence containment}, the order-sensitive analogue of the edge connectedness of \cref{sec:hypergraph-stability}.

\begin{definition} \label{def:sequence-contained}
    Let $G = (X,E_G)$ be a sequence hypergraph, and let $s = (x_0,\dots,x_n)$ be an ordered subset of $X$.
    A subset $S = \{e_1,\dots,e_m\} \subseteq E_G$ is \emph{connected} if there exists an ordering of its elements, again denoted $e_1,\dots,e_m$, with $e_i \cap e_{i+1} \neq \varnothing$ for all $1 \leq i \leq m-1$.
    We say that $s$ is \emph{sequence contained} in $S$, and that $S$ \emph{covers} $s$, if there is such an ordering for which, in addition:
    \begin{enumerate}
        \item $\{x_0,\dots,x_n\} \subseteq \bigcup_{i=1}^m e_i$;
        \item for each $i$, the overlap $s_i = s \cap e_i$, with the order inherited from $s$, is an ordered subsequence of the edge $e_i$;
        \item for each $1 \leq i \leq m-1$, $s_i \cap s_{i+1} \neq \varnothing$.
    \end{enumerate}
\end{definition}

The following lemma is the crux of this section: it is where the specific combinatorics of sequence hypergraphs enters, and it is exactly the composition axiom that makes the general framework apply.

\begin{lemma} \label{lem:sequence-covering-relation}
    Sequence containment is a covering relation on $\mathbb{S}(X)$ in the sense of \cref{def:covering-relation}.
    That is, writing ``$S$ covers $s$'' for ``$s$ is sequence contained in $S$'':
    \begin{enumerate}
        \item \textup{(Unit)} $\{s\}$ covers $s$;
        \item \textup{(Composition)} if $S = \{t_1, \dots, t_m\}$ covers $e$ and each $t_j$ is covered by $S_j \subseteq E_G$, then $\bigcup_{j=1}^m S_j$ covers $e$.
    \end{enumerate}
\end{lemma}
\begin{proof}
    The unit axiom is immediate.
    For composition, set $S = \bigcup_{j=1}^m S_j$.
    Since $\{t_1,\dots,t_m\}$ covers $e$ we have $e \subseteq \bigcup_j t_j$, and since each $S_j$ covers $t_j$ we have $t_j \subseteq \bigcup_{a \in S_j} a$; hence
    \[
        e \subseteq \bigcup_{j=1}^m t_j \subseteq \bigcup_{a \in S} a.
    \]
    It remains to check that $S$ is connected.
    Since $\{t_1, \dots, t_m\}$ is connected, after reordering we may assume $t_j \cap t_{j+1} \neq \varnothing$ for every $1 \leq j \leq m-1$.
    For each such $j$, choose a vertex $x_j \in t_j \cap t_{j+1}$.
    Since $S_j$ covers $t_j$ there is an edge $a_j \in S_j$ with $x_j \in a_j$, and since $S_{j+1}$ covers $t_{j+1}$ there is an edge $a_{j+1} \in S_{j+1}$ with $x_j \in a_{j+1}$; therefore $a_j \cap a_{j+1} \neq \varnothing$.
    Thus each $S_j$ is linked to $S_{j+1}$, so $S$ is connected, and $S$ covers $e$.
\end{proof}

Applying \cref{con:generated-filtration} to this covering relation gives the induced weight
\[
    \overline{w}_{G}(s) = \inf\Big\{ \textstyle\sum_{e \in S} w_G(e) \;\Big|\; S \text{ covers } s \Big\}
\]
and the filtered sequence hypergraph \emph{generated by} $G$, with entry weight $l_G = \overline{w}_{G}$; each sublevel set is a set of ordered subsets of $X$, hence again a sequence hypergraph.

\begin{definition} \label{def:dhypergraph-distance}
    Let $G = (X,E_G,w_G)$ and $H = (Y,E_H,w_H)$ be weighted sequence hypergraphs.
    The \emph{sequence hypergraph distortion} of a correspondence $C \colon X \rightrightarrows Y$ is
    \[
        \disDhyper(C) = \sup_\sigma \big| \overline{w}_{G}(\pi_1(\sigma)) - \overline{w}_{H}(\pi_2(\sigma)) \big|,
    \]
    where $\sigma$ ranges over finite ordered lifts through $C$, and the \emph{sequence hypergraph distance} between $G$ and $H$ is
    \[
        \dDhyper{G}{H} = \tfrac{1}{2}\inf \{ \disDhyper(C) \mid C \in \Corr(X,Y)\}.
    \]
\end{definition}

\subsection{Step 2: The associated path complex} \label{subsec:dhyp-step2}

We recall how a sequence hypergraph gives rise to a path complex, and hence to path homology.

\begin{definition} \label{def:sequence-hypergraph-path}
    Let $G = (X,E_G,w_G)$ be a weighted sequence hypergraph and $q \geq 1$.
    The \emph{density-$q$ path complex} $P^q(G)$ is the weighted path complex on $X$ whose paths are those sequences $p = (x_0,\dots,x_n)$ in which any $q$ or fewer consecutive vertices form an ordered subsequence of some edge $e \in E_G$, with weight function
    \[
    w_{P^q(G)}(p) = \sup\big\{ \inf\{w_G(e) \mid b \text{ an ordered subsequence of } e \in E_G \} \;\big|\; b = (x_i, \dots, x_{i+q-1}) \text{ a } q\text{-block of } p \big\}.
    \]
    As in \cref{sec:hypergraph-stability}, a \emph{$q$-block} of $p$ is a subpath of $p$ on $q$ vertices.
\end{definition}

As in \cref{sec:hypergraph-stability}, when the induced weight $\overline{w}_{G}$ is used in place of $w_G$, each infimum is attained at $e = b$ and the weight simplifies to
    \[
    w_{P^q(G)}(p) = \sup\{ \overline{w}_{G}(b) \mid b \text{ a } q\text{-block of } p \}.
    \]

\subsection{Step 3: The induced path weight} \label{subsec:dhyp-step3}

On the completed object, and for $q \leq 2$, inducing the sequence-hypergraph weight and then converting agrees with converting and then inducing.
As in \cref{sec:hypergraph-stability}, neither Step 4 nor the stability theorem depends on this.

\begin{proposition} \label{prop:dhyp-path-idempotent}
    Let $G = (X,E_G,w_G)$ be a weighted sequence hypergraph.
    Then $\overline{w}_{P^q(\cmpl{G})} \leq w_{P^q(\cmpl{G})}$, with equality when $q \leq 2$.
\end{proposition}
\begin{proof}
    The argument for \cref{prop:hyp-idempotent} applies verbatim, only with ordered subsequences in place of subsets.
    Writing $w = w_{P^q(\cmpl{G})}$ and $q \leq 2$, a $q$-block of $p$ is a subpath of $p$ on at most two vertices, so by \cref{item:window-short} of \cref{lem:window-of-concatenation} each such block $b$ is a $q$-block of a single element of any cover $S$ of $p$; hence $w(p) \leq \sum_{s \in S} w(s)$ and $w$ is covering-subadditive.
    By \cref{lem:covering-idempotent} the induced weight $\overline{w}$ is the largest such weight below $w$, whence $\overline{w} = w$.
    For $q \geq 3$ equality fails, by the example of \cref{rmk:q-at-least-three} with each edge given an arbitrary order.
\end{proof}

The two hypergraph theories do not reduce to one another, as the following remark explains.

\begin{remark} \label{rmk:forgetful-lax}
    It is tempting to try to deduce the stability of one of the hypergraph theories from the other.
    Forgetting the ordering of each edge gives a functor $U$ from weighted sequence hypergraphs to weighted hypergraphs, sending $G = (X, E_G)$ to $(X, \{\underline{e} : e \in E_G\})$, where $\underline{e}$ is the underlying set of the edge $e$.
    However, $U$ does not commute with the induced weight, and so does not carry one stability result to the other.
    The reason is the order-sensitivity of sequence containment: an ordered subset $s$ is sequence contained in $S$ only if each overlap $s \cap e_i$ is an \emph{ordered subsequence} of the edge $e_i$, whereas the covering relation of \cref{def:hypergraph-connected} ignores order entirely.
    A family $S$ that covers $s$ in $U(G)$ may therefore fail to sequence contain $s$ in $G$, so the sequence induced weight dominates the undirected one,
    \[
        \overline{w}_{U(G)}(\underline{s}) \leq \overline{w}_{G}(s),
    \]
    with strict inequality in general.
    The two constructions thus generate genuinely different filtrations, and neither stability theorem specializes to the other.
    Nor can one impose orders to reverse the comparison: since the order on an edge need not arise by restricting any single order on $X$, there is no functor from hypergraphs to sequence hypergraphs inverting $U$.
\end{remark}

\subsection{Step 4: Distance comparison} \label{subsec:dhyp-step4}

\begin{proposition} \label{prop:dhyp-distance-comparison}
    Let $G = (X,E_G,w_G)$ and $H = (Y,E_H,w_H)$ be weighted sequence hypergraphs. Then
    \[
    \dpathcomplex{P^q(\cmpl{G})}{P^q(\cmpl{H})} \leq \dDhyper{G}{H}.
    \]
\end{proposition}
\begin{proof}
Let $C \colon X \rightrightarrows Y$ be a correspondence; we show $\dispath(C) \leq \disDhyper(C)$.
Let $\sigma=((x_0,y_0),\dots,(x_n,y_n))$ be an ordered lift through $C$, and set $p=\pi_1(\sigma)$ and $p'=\pi_2(\sigma)$.
Since the weight of $\cmpl{G}$ is already induced, \cref{def:sequence-hypergraph-path} gives
\[
    w_{P^q(\cmpl{G})}(p) = \sup_i \overline{w}_{G}(b_i), \qquad w_{P^q(\cmpl{H})}(p') = \sup_i \overline{w}_{H}(b'_i),
\]
where $b_i = (x_i,\dots,x_{i+q-1})$ and $b'_i = (y_i,\dots,y_{i+q-1})$ are the corresponding $q$-blocks.
For each $i$, the restriction $\sigma_i = ((x_i,y_i),\dots,(x_{i+q-1},y_{i+q-1}))$ is an ordered lift through $C$ with $\pi_1(\sigma_i) = b_i$ and $\pi_2(\sigma_i) = b'_i$, so
\[
    |\overline{w}_{G}(b_i) - \overline{w}_{H}(b'_i)| \leq \disDhyper(C).
\]
As $p$ has finitely many $q$-blocks, \cref{lem:sup-comparison} applies and gives
\[
    |w_{P^q(\cmpl{G})}(p) - w_{P^q(\cmpl{H})}(p')| = \Big| \sup_i \overline{w}_{G}(b_i) - \sup_i \overline{w}_{H}(b'_i) \Big| \leq \disDhyper(C).
\]
Taking the supremum over ordered lifts gives $\dispath(C) \leq \disDhyper(C)$, and the infimum over correspondences gives the claim.
\end{proof}

Feeding this into stability for path complexes yields stability for sequence hypergraphs.

\begin{theorem} \label{thm:dhyp-stability}
Let $G = (X,E_G,w_G)$ and $H = (Y,E_H, w_H)$ be weighted sequence hypergraphs, and let $k \in \mathbb{N}$. Then
\[
    \dbottleneck{\dgm_k(G)}{\dgm_k(H)} \leq 2\dDhyper{G}{H}.
\]
\end{theorem}
\begin{proof}
By definition, the persistent path homology of a sequence hypergraph is that of its density-$q$ path complex.
The claim is therefore \cref{cor:transfer}, whose hypothesis is \cref{prop:dhyp-distance-comparison}.
\end{proof}
\section{Stability of Persistent Path Homology of Digraphs} \label{sec:network-stability}
In this section we deduce stability of persistent path homology of digraphs, recovering the result of Chowdhury and M\'emoli \cite{chowdhury2018persistent}, and following the four steps outlined in \cref{subsec:common-pattern}.
As in the previous sections, filtered objects are primitive and weighted ones are a source of examples.

\begin{definition} \label{def:filtered-digraph}
A \emph{digraph} is a pair $G = (X, E_G)$, where $X$ is a vertex set and $E_G \subseteq X \times X$ is a set of directed edges.
A \emph{filtered digraph} is a family $G = \{G^\delta\}_{\delta \geq 0}$ of digraphs on a fixed vertex set $X$, with inclusions $G^\delta \hookrightarrow G^{\delta'}$ for $\delta \leq \delta'$.
Its \emph{entry weight} $l_G \colon X \times X \to \mathbb{R}_{\geq 0} \cup \{\infty\}$ is $l_G(a,b) = \inf\{\delta \geq 0 \mid (a,b) \in E_G^\delta\}$.
\end{definition}

\begin{definition} \label{def:weighted-digraph}
A \emph{weighted digraph} is a triple $G = (X, E_G, w_G)$, where $(X, E_G)$ is a digraph and $w_G \colon E_G \to \mathbb{R}_{\ge 0} \cup \{ \infty \}$ is a weight function.
\end{definition}

\subsection{Step 1: The covering relation} \label{subsec:net-step1}

The covering relation for digraphs is generated by directed paths, the motivating example of \cref{sec:framework}.

\begin{definition} \label{def:digraph-covering}
    Let $G = (X,E_G)$ be a digraph and $(a,b) \in X \times X$.
    A subset $T \subseteq E_G$ \emph{covers} $(a,b)$ if it is a directed path from $a$ to $b$, that is, $T = \{(a,x_1),(x_1,x_2),\dots,(x_{m-1},b)\}$ for some vertices $x_1,\dots,x_{m-1} \in X$.
\end{definition}

This is a covering relation on $X \times X$ in the sense of \cref{def:covering-relation}: the unit axiom is given by the length-one path $(a,b)$, and composition by concatenation of directed paths.

Applying \cref{con:generated-filtration} to this covering relation gives the induced weight
\[
    \overline{w}_{G}(a,b) = \inf\Big\{ \textstyle\sum_{t \in T} w_G(t) \;\Big|\; T \text{ covers } (a,b) \Big\}
\]
and the filtered digraph \emph{generated by} $G$, with entry weight $l_G = \overline{w}_{G}$; each sublevel set is a set of vertex pairs, hence again a digraph.
This entry weight is exactly the shortest-path (Dijkstra) cost, and its idempotency reflects the familiar fact that shortest paths concatenate optimally.

\begin{definition} \label{def:network-distance}
    Let $G = (X,E_G,w_G)$ and $H = (Y,E_H,w_H)$ be weighted digraphs.
    The \emph{network distortion} of a correspondence $C \colon X \rightrightarrows Y$ is
    \[
        \disnetwork(C) = \sup \big\{ | \overline{w}_{G}(x,x') - \overline{w}_{H}(y,y') | \mid (x,y),(x',y') \in C \big\},
    \]
    and the \emph{network distance} between $G$ and $H$ is
    \[
        \dnetwork{G}{H} = \tfrac{1}{2}\inf \{ \disnetwork(C) \mid C \in \Corr(X,Y)\}.
    \]
\end{definition}

\subsection{Step 2: The associated path complex} \label{subsec:net-step2}

We recall how a digraph gives rise to a path complex, and hence to path homology.

\begin{definition} \label{def:digraph-path}
    Let $G = (X,E_G,w_G)$ be a weighted digraph.
    The \emph{associated path complex} $P(G)$, due to Grigor'yan, Lin, Muranov, and Yau \cite{grigor2012homologies, grigor2020pathcomplexes}, is the weighted path complex on $X$ whose paths are the directed paths of $G$, with weight function
    \[
    w_{P(G)}((x_0,\dots,x_n)) = \sup_{0 \leq i < n} w_G(x_i,x_{i+1}).
    \]
\end{definition}

Regarding a digraph as a sequence hypergraph whose edges all have size $2$, the path complex $P(G)$ coincides with the density-$2$ path complex $P^2(G)$ of \cref{sec:sequence-hypergraph-stability}: a sequence of vertices has all of its $2$-blocks among the edges of $G$ precisely when it is a directed path.
The two sections do not otherwise specialize to one another, as the following remark records.
As before, when the induced weight $\overline{w}_{G}$ is used in place of $w_G$, the path weight simplifies accordingly.

\begin{remark} \label{rmk:not-a-specialization}
    Although the path complexes agree, the covering relation of \cref{def:digraph-covering} is not the restriction of sequence containment, \cref{def:sequence-contained}, to edges of size $2$, and the two induced weights genuinely differ.
    In one direction, for the digraph with edges $(a,c)$ and $(a,b)$, the family $\{(a,c),(a,b)\}$ sequence contains $(a,b)$, since the overlaps $(a)$ and $(a,b)$ meet, but it is not a directed path from $a$ to $b$.
    In the other direction, a directed path of length at least two fails to sequence contain its endpoints, since its middle edges meet neither of them and condition (3) of \cref{def:sequence-contained} fails.
    For a concrete comparison, take $E_G = \{(a,x),(x,b),(a,b)\}$ with weights $1$, $1$, and $10$.
    The covering relation of \cref{def:digraph-covering} gives $\overline{w}_{G}(a,b) = 2$, realized by the path through $x$, whereas the only families sequence containing $(a,b)$ are those containing the edge $(a,b)$ itself, so sequence containment gives $\overline{w}_{G}(a,b) = 10$.
    For this reason the present section is developed independently of \cref{sec:sequence-hypergraph-stability}.
\end{remark}

\subsection{Step 3: The induced path weight} \label{subsec:net-step3}

On the completed object, inducing the digraph weight and then converting agrees with converting and then inducing.

\begin{proposition} \label{prop:net-path-idempotent}
    Let $G = (X,E_G,w_G)$ be a weighted digraph. Then $\overline{w}_{P(\cmpl{G})} = w_{P(\cmpl{G})}$.
\end{proposition}
\begin{proof}
    As in \cref{prop:hyp-idempotent}, write $w = w_{P(\cmpl{G})}$; since $P(G) = P^2(G)$, this is the case $q = 2$, where the equality holds.
    Let $S$ cover a path $p$.
    Each edge $(x_i,x_{i+1})$ of $p$ is a subpath of $p$ on two vertices, hence by \cref{item:window-short} of \cref{lem:window-of-concatenation} a subpath, and so an edge, of some $s \in S$, giving $\overline{w}_{G}(x_i,x_{i+1}) \leq w(s) \leq \sum_{s' \in S} w(s')$.
    Taking the supremum over the edges of $p$ gives $w(p) \leq \sum_{s \in S} w(s)$, so $w$ is covering-subadditive for the path-connected covering relation.
    By \cref{lem:covering-idempotent} the induced weight $\overline{w}$ is the largest such weight below $w$, whence $\overline{w} = w$.
\end{proof}

\subsection{Step 4: Distance comparison} \label{subsec:net-step4}

For digraphs the two distances agree exactly; this is what lets us recover the stability theorem of \cite{chowdhury2018persistent} in its original form.

\begin{proposition} \label{prop:net-distance-comparison}
    Let $G = (X,E_G,w_G)$ and $H = (Y,E_H,w_H)$ be weighted digraphs. Then
    \[
    \dpathcomplex{P(\cmpl{G})}{P(\cmpl{H})} = \dnetwork{G}{H}.
    \]
\end{proposition}
\begin{proof}
Let $C \colon X \rightrightarrows Y$ be a correspondence; we show $\dispath(C) = \disnetwork(C)$.

For $\dispath(C) \leq \disnetwork(C)$, let $\sigma = ((x_0,y_0),\dots,(x_n,y_n))$ be an ordered lift through $C$, and set $p = \pi_1(\sigma)$ and $p' = \pi_2(\sigma)$.
\Cref{def:digraph-path}, applied to $\cmpl{G}$, gives
\[
    w_{P(\cmpl{G})}(p) = \sup_{0 \leq i < n} \overline{w}_{G}(x_i,x_{i+1}), \qquad w_{P(\cmpl{H})}(p') = \sup_{0 \leq i < n} \overline{w}_{H}(y_i,y_{i+1}).
\]
For each $i$ we have $(x_i,y_i),(x_{i+1},y_{i+1}) \in C$, so $|\overline{w}_{G}(x_i,x_{i+1}) - \overline{w}_{H}(y_i,y_{i+1})| \leq \disnetwork(C)$.
As in \cref{prop:hyp-distance-comparison}, applying this to each term of the two suprema (finite in number) gives
\[
    |w_{P(\cmpl{G})}(p) - w_{P(\cmpl{H})}(p')| \leq \disnetwork(C).
\]
Taking the supremum over ordered lifts gives $\dispath(C) \leq \disnetwork(C)$.

For $\disnetwork(C) \leq \dispath(C)$, let $(x,y),(x',y') \in C$.
Then $\sigma = ((x,y),(x',y'))$ is an ordered lift through $C$ with $\pi_1(\sigma) = (x,x')$ and $\pi_2(\sigma) = (y,y')$, both paths of length one.
By \cref{def:digraph-path}, $w_{P(\cmpl{G})}((x,x')) = \overline{w}_{G}(x,x')$ and $w_{P(\cmpl{H})}((y,y')) = \overline{w}_{H}(y,y')$, so
\[
    |\overline{w}_{G}(x,x') - \overline{w}_{H}(y,y')| \leq \dispath(C).
\]
Taking the supremum over such pairs gives $\disnetwork(C) \leq \dispath(C)$.

Thus $\dispath(C) = \disnetwork(C)$ for every $C$, and taking the infimum over correspondences gives the claim.
\end{proof}

Feeding this into stability for path complexes yields stability for digraphs.

\begin{theorem} \label{thm:net-stability}
Let $G = (X,E_G,w_G)$ and $H = (Y,E_H, w_H)$ be weighted digraphs, and let $k \in \mathbb{N}$. Then
\[
    \dbottleneck{\dgm_k(G)}{\dgm_k(H)} \leq 2\dnetwork{G}{H}.
\]
\end{theorem}
\begin{proof}
By definition, the persistent path homology of a digraph is that of its associated path complex.
The claim is therefore \cref{cor:transfer}, whose hypothesis is the equality of \cref{prop:net-distance-comparison}.
\end{proof}
\section{Conclusion and future work} \label{sec:conclusion}

We have proven a stability theorem for the persistent path homology of path complexes (\cref{thm:path-stability}) and used it as a single engine from which stability for several combinatorial objects follows.
The mechanism is the covering framework of \cref{sec:framework}: a weight on the generating cells of an object is extended to all cells by covering, the extension is idempotent (\cref{lem:covering-idempotent}), and its sublevel sets give a filtered object whose stability reduces to that of path complexes.
Instantiating this framework gives stability of persistent path homology for hypergraphs (\cref{thm:hyp-stability}), sequence hypergraphs (\cref{thm:dhyp-stability}), and digraphs (\cref{thm:net-stability}), the last recovering the theorem of Chowdhury and M\'emoli \cite{chowdhury2018persistent}.
The parallel structure of these instances is summarized in \cref{tab:instances}, and \cref{tab:structure} records where each shared ingredient is found.

\begin{table}[ht]
\renewcommand{\arraystretch}{1.3}
\begin{center}
\begin{tabular}{@{}lccc@{}}
\toprule
& \textbf{Hypergraphs} & \textbf{Sequence hypergraphs} & \textbf{Digraphs} \\
& (\cref{sec:hypergraph-stability}) & (\cref{sec:sequence-hypergraph-stability}) & (\cref{sec:network-stability}) \\
\midrule
Filtered object & Def.~\ref{def:filtered-hypergraph} & Def.~\ref{def:sequence-hypergraph} & Def.~\ref{def:filtered-digraph} \\
Weighted object & Def.~\ref{def:weighted-hypergraph} & Def.~\ref{def:weighted-sequence-hypergraph} & Def.~\ref{def:weighted-digraph} \\
Covering relation & Def.~\ref{def:hypergraph-connected} & Def.~\ref{def:sequence-contained} & Def.~\ref{def:digraph-covering} \\
Distance & Def.~\ref{def:hypergraph-distance} & Def.~\ref{def:dhypergraph-distance} & Def.~\ref{def:network-distance} \\
Associated path complex & Def.~\ref{def:density-q-path} & Def.~\ref{def:sequence-hypergraph-path} & Def.~\ref{def:digraph-path} \\
Induced path weight & Prop.~\ref{prop:hyp-idempotent} & Prop.~\ref{prop:dhyp-path-idempotent} & Prop.~\ref{prop:net-path-idempotent} \\
Distance comparison & Prop.~\ref{prop:hyp-distance-comparison} & Prop.~\ref{prop:dhyp-distance-comparison} & Prop.~\ref{prop:net-distance-comparison} \\
Stability theorem & Thm.~\ref{thm:hyp-stability} & Thm.~\ref{thm:dhyp-stability} & Thm.~\ref{thm:net-stability} \\
\bottomrule
\end{tabular}
\end{center}
\caption{The ingredients shared by \cref{sec:hypergraph-stability,sec:sequence-hypergraph-stability,sec:network-stability}.}
\label{tab:structure}
\end{table}

Several directions remain open.

\paragraph{Metric spaces and the Vietoris--Rips complex.}
A finite metric space gives a filtered simplicial complex, and it is natural to ask whether the classical stability of Vietoris--Rips persistent homology \cite{chazal-desilva-oudot:stability-geometric-complexes} can be recovered from \cref{thm:path-stability}.
Passing through the face digraph of \cite{grigor2012homologies} gives the right persistence module, by the comparison of \cite{grigor2012homologies} between the path homology of a face digraph and the simplicial homology of the complex, but the wrong distance: \cref{def:path-distance} then quantifies over correspondences between the simplices of the two complexes rather than between their points, and these can fail to have finite distortion.
Already for a one-point and a two-point metric space the only such correspondence relates the single simplex of the first to two $\subseteq$-incomparable simplices of the second, which forces the distortion to be infinite.
A path complex built directly on the points, with a path allowed once its vertex set has small enough diameter, has the right distance but is not known to have the right homology.
Identifying a path-complex model of the Vietoris--Rips filtration for which both hold would bring simplicial persistent homology into the framework developed here.

\paragraph{Other homology theories of hypergraphs.}
The path homology of hypergraphs used here is one of many homology theories that have been attached to hypergraphs.
The survey of Gasparovic, Purvine, Sazdanovi\'c, Wang, Wang, and Ziegelmeier \cite{gasparovic-purvine-sazdanovic-wang-wang-ziegelmeier:survey} describes nine such constructions.
Among them is the embedded homology of Bressan, Li, Ren, and Wu \cite{bressan-li-ren-wu:embedded}, built from the largest chain complex contained in, and the smallest containing, the span of the hyperedges.
It would be natural to ask which of these theories admit stability results, and whether any of them fit a covering framework of the kind developed here.

\paragraph{Directed and oriented hypergraphs.}
The sequence hypergraphs of \cref{sec:sequence-hypergraph-stability} are one way to add directional structure to a hypergraph, in which each edge carries a linear order on its vertices.
Two other notions appear in the literature and are not captured by our setup.
A \emph{directed hypergraph} has edges that are ordered pairs $(T, H)$ of disjoint vertex sets, a \emph{tail} and a \emph{head}, modelling a many-to-many transition; these arise in database theory and in the modelling of chemical reaction networks.
An \emph{oriented hypergraph} instead equips each vertex--edge incidence with a sign, generalizing the incidence structure of a signed graph.
Persistent and embedded homologies have been studied for hyperdigraphs by assigning orientations to hyperedges; developing stability results for these directed variants, and relating them to the theory here, would clarify how directional information interacts with the covering framework.

\paragraph{A more abstract framework.}
Our covering framework is stated concretely, one instance at a time, and the reduction to path complexes is carried out separately in each section.
It would be worthwhile to formulate the framework abstractly enough that the four steps become a single theorem, with each object obtained by supplying a covering relation on a set of cells and a functor to path complexes.
Such a formulation would likely clarify the relationship between our covering relations and the Grothendieck pretopologies discussed in \cref{subsec:coverage-literature}, and might allow objects beyond those treated here to be brought under the same stability result with no further work.

\bibliographystyle{amsalphaurlmod}
\bibliography{all-refs}

%Uncomment the following if you would like an appendix
\appendix
\renewcommand{\thesection}{\Alph{section}}
%\begin{appendices}
%\section{Code for computing monoidal products}
%  \lstinputlisting[language=Python]{Graph Products.py}
%\end{appendices}
% End of appendix
%\newpage
% Uncomment the following if you have a bibliography file

\end{document}